\documentclass[11pt, reqno]{amsart}
\usepackage[T1]{fontenc}
\usepackage{amssymb,amsthm,amsmath}
\usepackage{enumitem,footnote,color,caption,tikz}
\usepackage[colorlinks=True]{hyperref}

\evensidemargin\oddsidemargin

\newtheorem{thm}{Theorem}
\newtheorem{lem}[thm]{Lemma}
\newtheorem{prop}[thm]{Proposition}
\newtheorem{cor}[thm]{Corollary}
\newtheorem{defn}[thm]{Definition}
\newtheorem{ex}{Example}

\DeclareMathOperator{\Jac}{Jac}
\DeclareMathOperator{\Ima}{Im}
\DeclareMathOperator{\Id}{Id}

\newcommand{\si}{\sigma}
\newcommand{\R}{\mathbb{R}}
\newcommand{\N}{\mathbb{N}}
\newcommand{\D}{\Delta_{\si_{d-1}}}
\newcommand{\la}{\lambda}
\newcommand{\al}{\alpha}
\newcommand{\Z} {\mathbb{Z}}

\renewcommand{\leq}{\leqslant}
\renewcommand{\geq}{\geqslant}

\begin{document}

\baselineskip=15pt

\title[Enumeration of walks in multidimensional orthants]{Enumeration of walks in multidimensional orthants and reflection groups}

\date{\today}
\author{Léa Gohier}

\address{Université de Tours, Institut Denis Poisson, Tours, France. {\em Email}: {\tt lea.gohier@univ-tours.fr}}

\author{Emmanuel Humbert}

\address{Université de Tours, Institut Denis Poisson, Tours, France. {\em Email}: {\tt emmanuel.humbert@univ-tours.fr}}

\author{Kilian Raschel} 

\address{CNRS, International Research Laboratory France-Vietnam in mathematics and its applications, Vietnam Institute for Advanced Study in Mathematics, Hanoï, Vietnam. {\em Email}: {\tt raschel@math.cnrs.fr}}



\begin{abstract} 
We consider (random)\ walks in a multidimensional orthant. Using the idea of universality in probability theory, one can associate a unique polyhedral domain to any given walk model. We use this connection to prove two sets of new results. First, we are interested in a group of transformations naturally associated with any small step model; as it turns out, this group is central to the classification of walk models. We show a strong connection between this group and the reflection group through the walls of the polyhedral domain. As a consequence, we can derive various conditions for the combinatorial group to be infinite. Secondly, we consider the asymptotics of the number of excursions, whose critical exponent is known to be computable in terms of the eigenvalue of the above polyhedral domain. We prove new results from spectral theory on the eigenvalues of polyhedral nodal domains. We believe that these results are interesting in their own right; they can also be used to find new exact asymptotic results for walk models corresponding to these nodal polyhedral domains.
\end{abstract}

\maketitle

\section{Introduction and main results}
\label{sec:introduction}

\subsection*{Lattice walks in multidimensional orthants}
A lattice walk is a sequence of points $P_0, P_1,\ldots ,P_n$ of $\mathbb Z^d$, $d\geq 1$. The points $P_0$ and $P_n$ are its starting and end points, respectively, the consecutive differences $P_{i+1}-P_i$ its steps, and $n$ is its length. Given a set $\mathcal S\subset \mathbb Z^d$, called the step set, a set $C\subset \mathbb Z^d$ called the domain (which in this paper will systematically be the cone $\mathbb R_+^d$, called the $d$-dimensional orthant), and elements $P$ and $Q$ of $C$, we are interested in the number 
\begin{equation*}
   e_C(P,Q; n)    
\end{equation*}
of (possibly weighted)\ walks (or excursions)\ of length $n$ that start at $P=P_0$, have all their steps in $\mathcal S$, have all their points in $C$, and end at $Q=P_n$. See Figure~\ref{fig:ex}. Normalising the weights with the condition that they sum to one, we obtain transition probabilities, and the number $e_C(P,Q; n)$ can be interpreted as the probability that a random walk starting at $P$ will reach the point $Q$ at time $n$ while remaining in the domain $C$.

In the last twenty years, there has been a dense research activity in the mathematical community on the enumerative aspects of walks confined to cones, in particular to the $d$-dimensional orthant. To summarise, three main questions have attracted most attention: the first is to determine, if possible, a closed-form formula for the number of walks $e_C(P,Q; n)$. Of course, such an explicit formula is not expected to exist in general, and in most cases can be explained by bijections with other combinatorial objects. The second question concerns the asymptotic behaviour, e.g.\ of the number of excursions $e_C(P,Q; n)$, in the regime where the length $n\to\infty$, while the start and end points remain fixed: as we will see later, in general one has
\begin{equation}
\label{eq:one-term-asymp}
   e_C(P,Q;n) \sim c(P,Q) \frac{\rho^n}{n^{\alpha}},
\end{equation}
for some quantities $c,\rho,\alpha$ depending on the model; $\rho$ is called the structural constant, it describes the exponential growth of the number of excursions, while $\alpha$ is called the critical exponent. The last question would focus more on the complexity of generating functions associated with these models, such as the excursion series
\begin{equation}
\label{eq:excursions_series}
   \sum_{n\geq0} e_C(P,Q;n)t^n.
\end{equation}
One would then ask whether the last power series satisfies any algebraic or differential equation. The answer to the last problem may allow us to classify the models according to the complexity of their generating function, and is further related to the first two questions.

\begin{figure}[t!]
\begin{center}
\includegraphics[height=3.7cm]{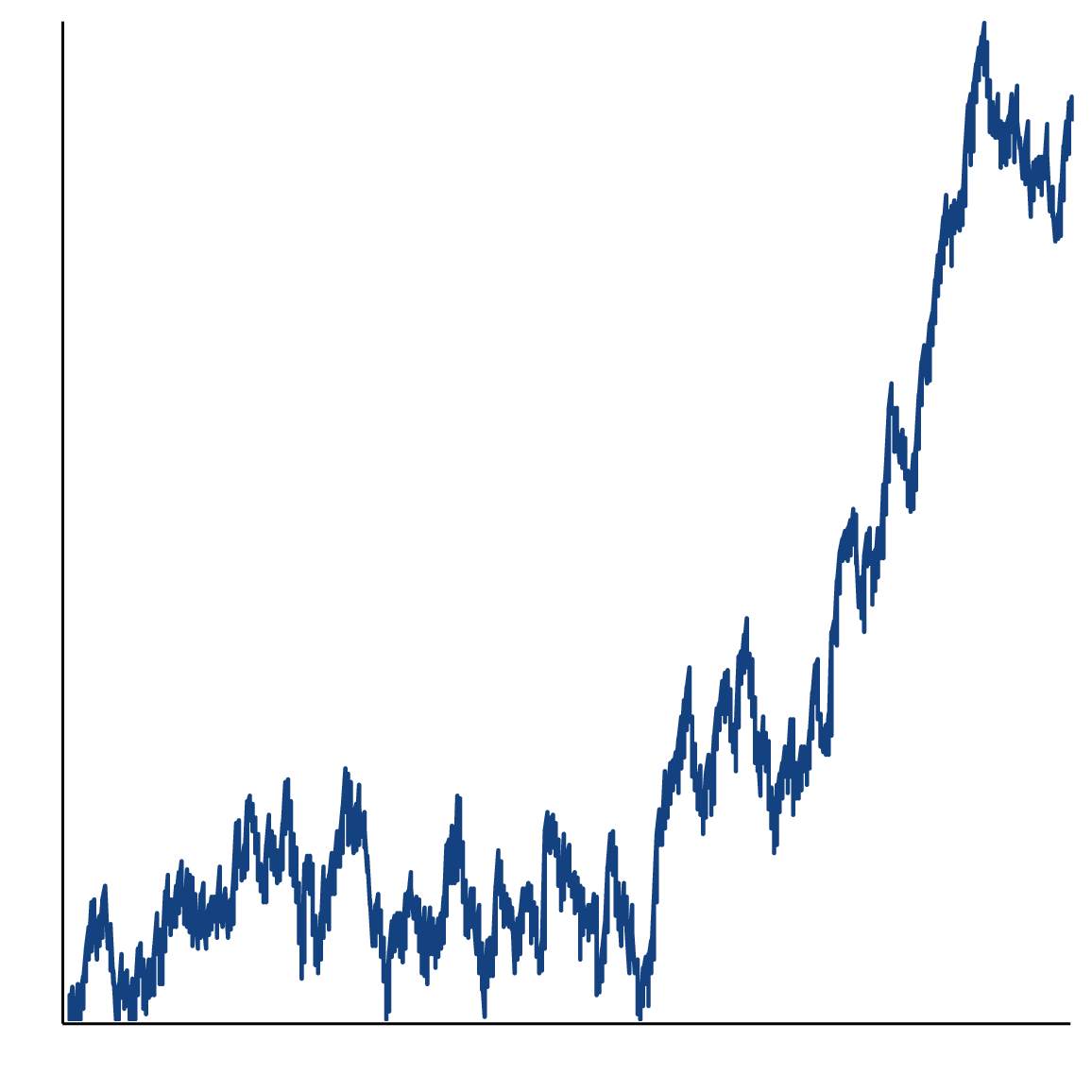}\qquad\qquad
\includegraphics[height=3.7cm]{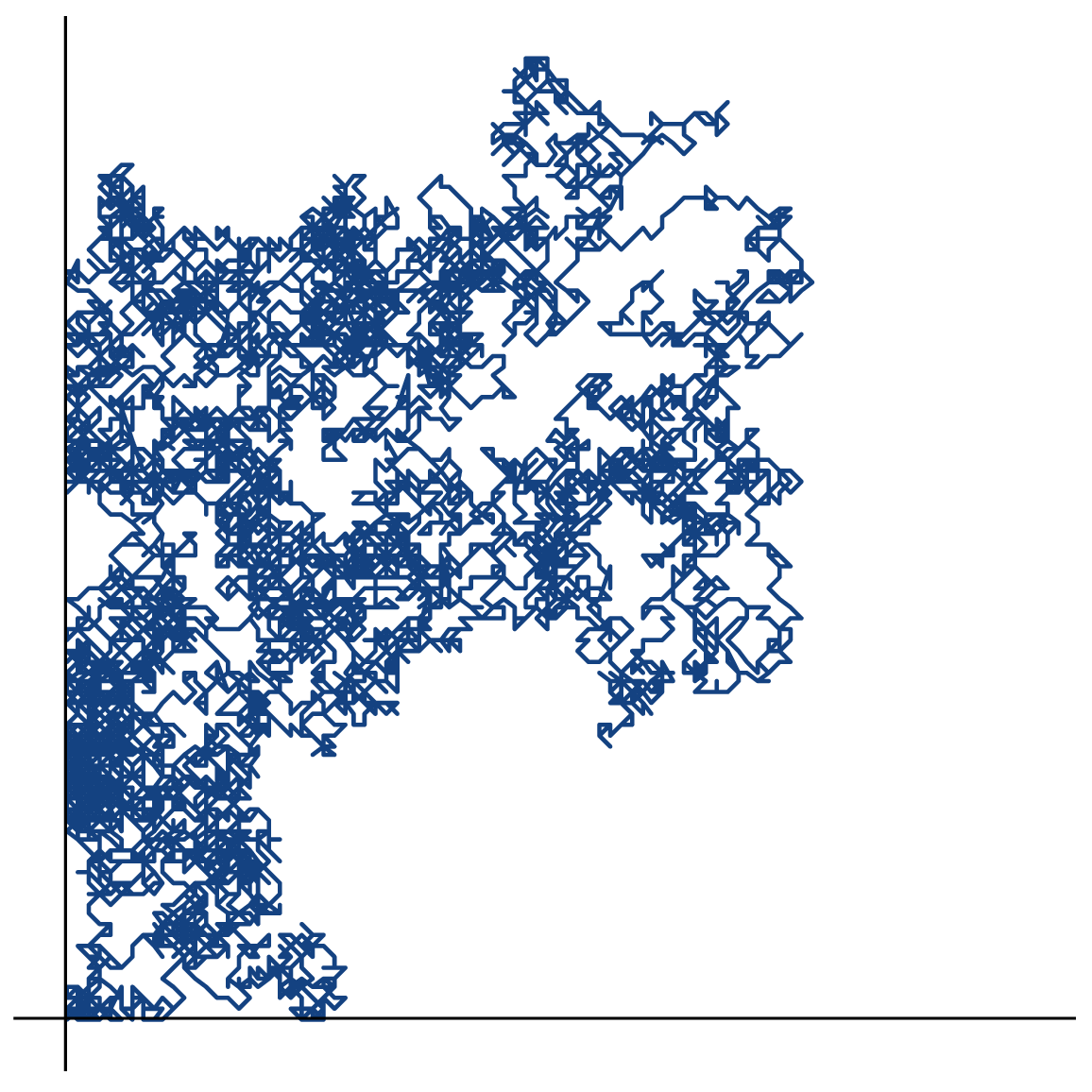}\qquad\qquad
\includegraphics[height=3.7cm]{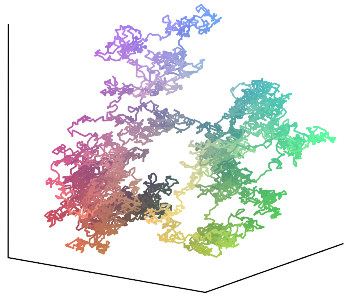}
\end{center}
\caption{Walks in orthants $\mathbb R_+^d$ in small dimensions $d=1,2,3$}
\label{fig:ex}
\end{figure}

Historically, these questions were first addressed for one-dimensional models. In this case, the model is that of walks on the positive half-line, which has a long tradition in the probabilistic community \cite{BeDo-94} and also in the combinatorial community \cite{BaFl-02}. Note that in dimension one, walks with bounded jumps admit algebraic generating functions \cite{BaFl-02}; we can thus systematically solve the three questions mentioned above. The case of dimension two gives rise to the model of walks in the quarter plane, which has been studied intensively, see \cite{BMMi-10,BoRaSa-14,DrHaRoSi-18} and references therein. The variety of possible behaviours and answers to the three questions is much richer than in dimension one; however, a combination of techniques from different fields (probability \cite{BoRaSa-14,DeWa-15}, complex analysis \cite{FaIaMa-17}, analytic combinatorics in several variables \cite{BMMi-10,MeMi-16,MeWi-19}, Galois theory of difference equations \cite{DrHaRoSi-18}, etc.)\ eventually lead to a deep understanding of this class of models, at least in the case of small steps (this hypothesis means that the walk can only go to neighbours at $\ell^\infty$-distance one). Walk models in higher dimension $d\geq 3$ have been less studied. The variety of behaviours seems to increase dramatically with dimension, and exactly solvable models become an exception. See the recent papers \cite{BoBMKaMe-16,DuHoWa-16,KaWa-17,BoPeRaTr-20,HiJeRa-24} for properties of three-dimensional models and \cite{BuHoKa-21} for an approach in dimension four. In higher dimensions, some models are equivalent to non-intersecting paths \cite{DoOC-05,LeLePe-12,Fe-14,Fe-18}, which are well understood and studied in both the combinatorial and physics literature.

\subsection*{Focus on two important tools}
In order to present the main results of this paper, we recall two different tools that are of crucial interest in obtaining some of the previous results in the literature. The first is a group introduced in dimension two in the combinatorial context in \cite{BMMi-10}, following the idea of Fayolle, Iasnogorodski and Malyshev in \cite{FaIaMa-17}, see also \cite{Ma-71}. This group will be properly defined later in the paper (see Section~\ref{sec:preliminaries}, in particular \eqref{eq:group_G_def}), but can be presented informally as follows. It is a symmetry group of involutions defined by the steps of the model $\mathcal S$. The main application is that, if it is finite, its action on a functional equation naturally associated with the model can lead to explicit expressions for the generating functions \eqref{eq:excursions_series} (with some further information on the asymptotics \eqref{eq:one-term-asymp} and algebraic complexity). In principle, this method works in dimension two \cite{FaIaMa-17,BMMi-10} and higher \cite{BoBMKaMe-16,Ya-17,BuHoKa-21}. However, there is no criterion to decide whether the group is finite (even in dimension two), nor any interpretation of the group using more classical groups. 

The second tool comes from probability theory and is the asymptotics written in \eqref{eq:one-term-asymp}, originally derived in \cite[Eq.~(12)]{DeWa-15},
where the prefactor $c(P,Q)$ depends on the start and end points (and can be interpreted as a discrete harmonic function), the structural constant $\rho$ is easily computed in terms of the parameters, see \eqref{eq:formula_rho}, and the critical exponent is computed by
Denisov and Wachtel in \cite{DeWa-15}:
\begin{equation}
\label{eq:DW_exponent}
    \alpha = 1+\sqrt{\lambda_1+\left(\tfrac{d}{2}-1\right)^2},
\end{equation}
where $\lambda_1$ is the principal eigenvalue for a Dirichlet problem on a subdomain of the sphere $\mathbb S^{d-1}$, see our Theorem~\ref{thm:DW_formula_exponent} for a more precise statement. See \cite{DB-86,BaSm-97} for probabilistic references where the exponent \eqref{eq:DW_exponent} appears in the context of Brownian motion in cones; see also \cite{Va-99}. While the formula \eqref{eq:DW_exponent} characterises the critical exponent in \eqref{eq:one-term-asymp}, it is not clear a priori for which walk models $\lambda_1$ can be computed in closed form.

\subsection*{A polyhedral domain depending on the walk model}
In this paper we introduce a new idea, based on reflection groups and more generally Coxeter groups, to study at once the combinatorial group and the critical exponent mentioned above. The main tool is to transform the orthant $\mathbb R_+^d$, which is the natural confinement domain of the walk (see again Figure~\ref{fig:ex}), into another polyhedral domain obtained as a linear transformation of the orthant, see Figure~\ref{fig:action_Delta}. More precisely, the new domain is 
\begin{equation}
    \label{eq:def_new_domain}
    \Delta^{-\frac{1}{2}}\mathbb R_+^d,
\end{equation}
where the linear map $\Delta$ allowing this transformation is canonical and given by the idea of universality in probability theory, as we now explain. 

\begin{figure}[t!]
    \centering
    \includegraphics[width=3cm]{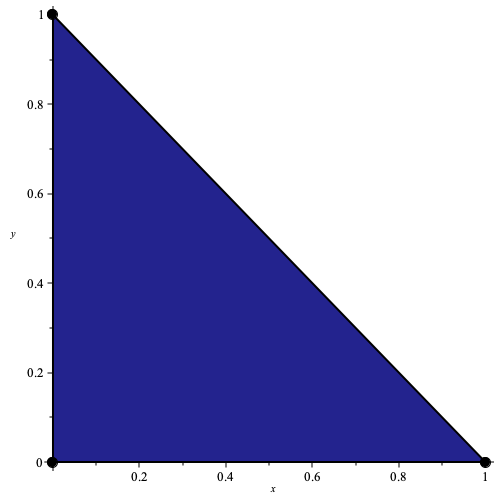}\quad
    \includegraphics[width=3cm]{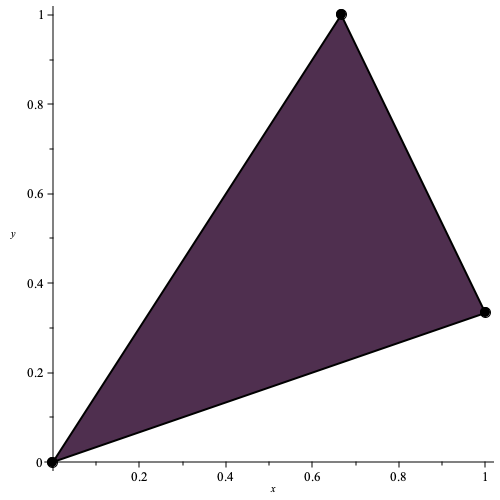}\quad
    \includegraphics[width=4cm]{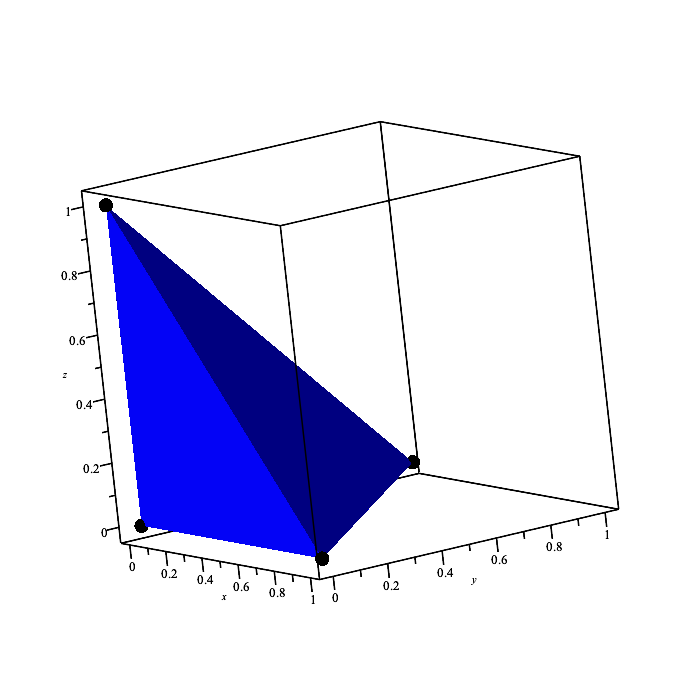}\quad
    \includegraphics[width=4cm]{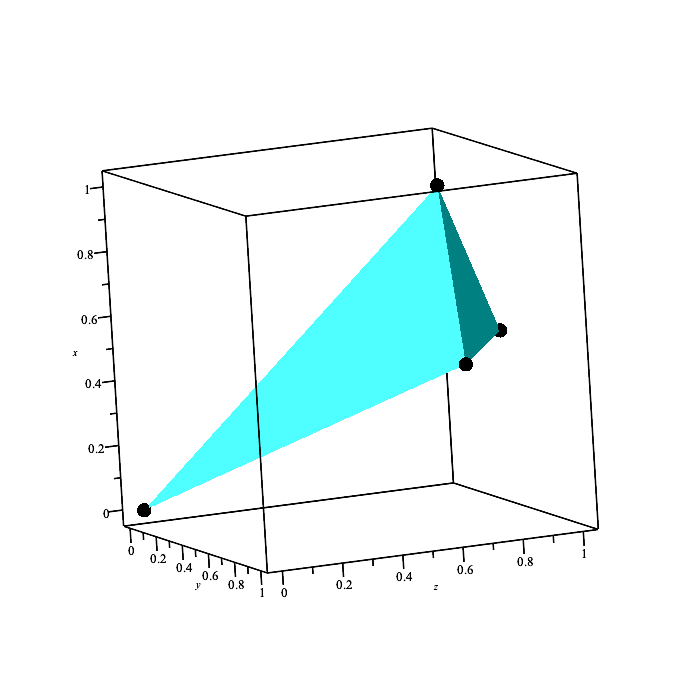}
    \caption{The orthant $\mathbb R_+^d$ is mapped into the polyhedral domain (or pyramid)\ \eqref{eq:def_new_domain}, which depends on the model (illustration in dimension two and three). In the new domain, the model has a covariance matrix equal to the identity.}
    \label{fig:action_Delta}
\end{figure}

For a given model $\mathcal S$, if $w(s)$ denotes the weight of the step $s=(s_1,\ldots,s_d)\in\mathcal S$, the drift is defined as the vector $\sum_{s\in\mathcal S}w(s) s$ and the covariance matrix is the $d\times d$ matrix with coefficients $\sum_{s\in\mathcal S}w(s) s_is_j$. 
A generic walk model will have a non-zero drift and a covariance different from the identity. It is then natural to perform an exponential transformation of the weights to remove the drift and transform the covariance matrix into the identity matrix. This change of measure is classical in probability theory and is known as the Cram\'er transformation (it will be recalled in Section~\ref{sec:preliminaries}, see \eqref{eq:formula_new_weights}). In this way, the new random walk is in the region of attraction of a standard Brownian motion, and various universal limit theorems can be applied, in particular the results of Denisov and Wachtel \cite{DeWa-15}. 

If we denote by $\Delta$ the covariance matrix of the new random walk $(Y_n)_{n\geq 0}$ after changing the transition probabilities, the modified random walk $(\Delta^{-\frac{1}{2}}Y_n)_{n\geq 0}$ has identity covariance matrix. The counterpart is that the confinement domain of the process is no longer the orthant $\mathbb R_+^d$ but the polyhedral domain \eqref{eq:def_new_domain}, which appears naturally in this way.
See Section~\ref{sec:preliminaries}, in particular equation~\eqref{eq:formula_new_weights}, for mode details of this construction.

\subsection*{Main results}

Given the polyhedral domain \eqref{eq:def_new_domain} it is natural to ask two main questions. First, one can define the reflection group $H$ spanned by the reflections through the sides of the domain, and ask how this group compares to the natural combinatorial group $G$ mentioned above. This question is at the origin of our first set of results (Part~\ref{part:1} of the paper). 

In Theorem~\ref{thm:prop_scalar_product} and Corollary~\ref{cor:imJH} we construct a surjective morphism $G \to H$ which sends the set of generators of $G$ to the set of generators of $H$. 
This clarifies the connection between the group $G$ and $H$, and also between $G$ and Coxeter groups. A direct consequence (Corollary~\ref{GinHfin})\ is that if $H$ is infinite, then $G$ should also be infinite. More generally, our results extend and simplify the strategy developed in \cite{BMMi-10,DuHoWa-16,KaWa-17} to prove that some models $\mathcal S$ admit an infinite group $G$ in dimension two and three (which is based on showing that the order of the composition of two generators of $G$ is infinite). In our case, since the composition of two reflections is simply a rotation of the angle twice the angle between two boundary hyperplanes of the polyhedral domain, its order can be read directly from the covariance matrix.
Another relevant remark is that $H$ is a reflection group and thus a Coxeter group. If it is finite, it must belong to a short list of known examples (see e.g.\ \cite{Hu-90}), and thus it suffices to exclude these cases to prove that $H$ (and thus $G$)\ is infinite. See Section~\ref{sec:app} for more details and various examples. 

The second set of results (Part~\ref{part:2} of this paper)\ concerns the critical exponent $\alpha$ in \eqref{eq:one-term-asymp}, or equivalently the principal eigenvalue $\lambda_1$.
There are basically two ways to compute $\alpha$. In some cases there is an explicit expression for the number $e_C(P,Q;n)$ (or for the associated generating function \eqref{eq:excursions_series}), which can be analysed asymptotically. See for example \cite{BMMi-10,BoBMKaMe-16} for such calculations. However, this is very rare and is related to bijections with other combinatorial models or obtained by delicate combinatorial manipulations.  Forgetting any combinatorial interpretation of the numbers $e_C(P,Q;n)$, the second approach would be to study the quantity $\lambda_1$ directly from a geometric perspective. 
In dimension two one can fully compute the eigenvalue $\lambda_1$, which is given by $\bigl(\frac{\pi}{\arccos (-a)}\bigr)^2$, where $a$ is a correlation coefficient, which can be computed starting from the parameters. See Example~2 in \cite{DeWa-15} and \cite{BoRaSa-14} for detailed computations in dimension two. In dimension $d\geq 3$ there is no hope of finding such a general formula for $\lambda_1$. This is illustrated in \cite{BoPeRaTr-20} in dimension three, where $\lambda_1$ is characterised as the first eigenvalue of spherical triangles, which is generally not computable in closed form. The only exceptions are triangles belonging to a small list of examples given in \cite{BeBe-80} and \cite[Sec.~3]{Be-83}, which correspond to certain crystallographic groups. In these cases it is possible to calculate the principal eigenvalue and in fact the whole spectrum. 

Surprisingly, Bérard and Besson's works \cite{BeBe-80,Be-83} have not been extended to the case of arbitrary dimension (although a list of relevant domains is proposed in \cite{ChSo-09} in dimension four). This is the main result of our second part, see Theorem~\ref{pndomains}. More precisely:
\begin{itemize}
    \item We compute the Dirichlet eigenvalue $\la_1$ of polyhedral domains (in the sense of  \eqref{eq:def_new_domain}) which are also nodal domains   of ${\mathbb{S}}^{d-1}$, i.e.\ for which there exists $\phi \in \mathcal C^{\infty}({\mathbb{S}}^{d-1})$ such that $\phi$ is an eigenfunction of the Laplacian satisfying $\phi>0$ on the domain and 
$\phi=0$ on the boundary. This part of the statement in Theorem~\ref{pndomains} is exactly the $d$-dimensional generalisation of \cite{BeBe-80,Be-83}. The main technical novelty is that we find an expression for the eigenfunction which makes possible the computation of $\la_1$ in all dimensions.
    \item We classify  all such domains. We show that they must be the intersection of a chamber of a finite Coxeter group with $\mathbb S^{d-1}$.
\end{itemize}
This provides new examples of walk models, in any dimension, for which one can compute the asymptotics \eqref{eq:one-term-asymp}.

\part{The combinatorial group as a reflection group}
\label{part:1}

The first part consists of three sections. First, in Section~\ref{sec:preliminaries}, we introduce our notation and recall some details about the connection between the asymptotics of the excursion and the eigenvalue $\lambda_1$. Section~\ref{sec:surjective} contains our main theoretical results, in particular Theorem~\ref{thm:prop_scalar_product} and Corollary~\ref{cor:imJH} on the surjective morphism $G \to H$. Finally, in Section~\ref{sec:app} we propose various applications of our results, giving several examples and techniques to prove that $G$ is infinite or to show that $G$ and $H$ are isomorphic. 

\section{Preliminary notations and results}
\label{sec:preliminaries}

\subsection{Assumptions on the step set}
Define the inventory of the model $\mathcal S$ as follows:
\begin{equation}
\label{eq:inventory}
    \chi_\mathcal S(x_1,\ldots,x_d) = \sum_{(i_1,\ldots,i_d)\in\mathcal S} w(i_1,\ldots,i_d)x_1^{i_1}\cdots x_d^{i_d},
\end{equation}
where $w(i_1,\ldots,i_d)>0$ is the weight of the step $(i_1,\ldots,i_d)\in\mathcal S$. We will normalize the weights in such a way that $\chi_\mathcal S(1,\ldots,1)=1$, so that they are also  transition probabilities. Most of the time we shall assume that the step set satisfies the following irreducibility assumption:
\begin{enumerate}[label={\rm (H\arabic*)},ref={\rm (H\arabic*)}]
     \item\label{it:hypothesis_irreducible}
     For any two points $P,Q$ in the domain $C$, the set $\{n\in\mathbb N : e_C(P,Q;n)\neq 0\}$ is non-empty, with $\mathbb N=\{1,2,\ldots\}$.
\end{enumerate}
As a consequence of \ref{it:hypothesis_irreducible}, the walk can visit any point in the domain $C$, independent of its starting point.

\subsection{The group of the model}
This group was first introduced in the context of two-dimensional walks \cite{Ma-71,FaIaMa-17,BMMi-10} and turns out to be very useful. Let $\chi_\mathcal S$ be the inventory \eqref{eq:inventory}. Introduce the notation
\begin{align*}
     \chi_\mathcal S(x_1,\ldots,x_d) &=x_1A_{1}(x_2,\ldots,x_d)+B_1(x_2,\ldots,x_d)+\overline{x_1}C_{1}(x_2,\ldots,x_d)\\
                     &=x_2A_{2}(x_1,x_3,\ldots,x_d)+B_2(x_1,x_3,\ldots,x_d)+\overline{x_2}C_{2}(x_1,x_3,\ldots,x_d)\\
                     &=\cdots\\
                     &=x_dA_{d}(x_1,\ldots,x_{d-1})+B_d(x_1,\ldots,x_{d-1})+\overline{x_d}C_{d}(x_1,\ldots,x_{d-1}),
\end{align*}
where $\overline{x_i}=\frac{1}{x_i}$. Under the assumption~\ref{it:hypothesis_irreducible}, the step set has a positive step in each direction and $A_{1},\ldots,A_d$ are all non-zero. By definition, the group of $\mathcal S$ is the group 
\begin{equation}
    \label{eq:group_G_def}
    G=\langle\varphi_1,\ldots,\varphi_d\rangle
\end{equation}
of birational transformations of the variables $[x_1,\ldots,x_d]$ generated by the following involutions:
\begin{equation}
\label{eq:expression_generators}
     \left\{\begin{array}{rl}
     \varphi_1([x_1,\ldots,x_d]) & =\, \left[\overline{x_1}\frac{C_{1}(x_2,\ldots,x_d)}{A_{1}(x_2,\ldots,x_d)},x_2,\ldots,x_d\right],\smallskip\\
     \varphi_2([x_1,\ldots,x_d]) & =\, \left[x_1,\overline{x_2}\frac{C_{2}(x_1,x_3,\ldots,x_d)}{A_{2}(x_1,x_3,\ldots,x_d)},x_3,\ldots,x_d\right],\smallskip\\
     \cdots & \\
     \varphi_d([x_1,\ldots,x_d]) & =\, \left[x_1,\ldots,x_{d-1},\overline{x_d}\frac{C_{d}(x_1,\ldots,x_{d-1})}{A_{d}(x_1,\ldots,x_{d-1})}\right].
     \end{array}\right.
\end{equation}
The generators of the group $G$ therefore satisfy the relations $\varphi_1^2=\cdots=\varphi_d^2=\Id$, plus possible other relations, depending on the model.

\subsection{Probabilistic estimates}

We need to introduce the following quantity. Given a cone $T$ in $\mathbb R^d$ (which in our case will be a polytope, or pyramid, defined by an intersection of $d$ linear half-spaces, see for example \eqref{eq:def_T}), we
define $\lambda_1$ as the smallest eigenvalue $\Lambda$ of the Dirichlet problem for the Laplace-Beltrami operator $\Delta_{\mathbb S^{d-1}}$ on the sphere $\mathbb S^{d-1}\subset\mathbb R^d$
\begin{equation}
\label{eq:Dirichlet_problem}
     \left\{
\begin{array}{rll}
     -\Delta_{\mathbb S^{d-1}}m&=\ \Lambda m & \text{in } T \cap \mathbb S^{d-1},\\
     m&=\ 0& \text{in } \partial \bigl(T \cap \mathbb S^{d-1}\bigr).
     \end{array}
     \right.
\end{equation}
See the book \cite{Ch-84} for general properties of the eigenvalues of the above Dirichlet problem.

\begin{thm}[\cite{DeWa-15,DeWa-19}]
\label{thm:DW_formula_exponent}
Let $\mathcal S$ be a step set satisfying~\ref{it:hypothesis_irreducible}, and let $\chi_\mathcal S$ be its inventory~\eqref{eq:inventory}. The system of equations
\begin{equation}
\label{eq:chi_critical_point}
     \frac{\partial \chi_\mathcal S}{\partial x_1}=\cdots = \frac{\partial \chi_\mathcal S}{\partial x_d}=0
\end{equation}
admits a unique solution in $(0,\infty)^d$, denoted by $\boldsymbol{x_0}$. Define
the covariance matrix 
\begin{equation}
\label{eq:expression_covariance_matrix}
     \Delta=\bigl(a_{i,j}\bigr)_{1\leq i,j\leq d},\quad \text{with } a_{i,j}=\frac{\frac{\partial^2 \chi_\mathcal S}{\partial x_i\partial x_j}(\boldsymbol{x_0})}{\sqrt{\frac{\partial^2 \chi_\mathcal S}{\partial x_i^2}(\boldsymbol{x_0})\cdot \frac{\partial^2 \chi_\mathcal S}{\partial x_j^2}(\boldsymbol{x_0})}}.
\end{equation}
Let $\Delta^{-\frac{1}{2}}$ denote the inverse of the symmetric, positive definite square root of the covariance matrix $\Delta$, see \eqref{eq:def_power_Delta}.
Consider the $d$-dimensional polytope 
\begin{equation}
    \label{eq:def_T}
    T=\Delta^{-\frac{1}{2}}\mathbb R_+^d.
\end{equation}
Let $\lambda_1$ be the smallest eigenvalue of the Dirichlet problem \eqref{eq:Dirichlet_problem}. Then for all $P,Q\in\mathbb N_0^d$ and $n$ such that $e_C(P,Q;n)\neq 0$, the asymptotics \eqref{eq:one-term-asymp} of the number of excursions going from $P$ to $Q$ holds, where 
\begin{equation}
\label{eq:formula_rho} 
     \rho =\min_{(0,\infty)^d} \chi_\mathcal S
\end{equation}
and the critical exponent $\alpha$ in \eqref{eq:one-term-asymp} is given by \eqref{eq:DW_exponent}, namely $\alpha = 1+\sqrt{\lambda_1+\left(\tfrac{d}{2}-1\right)^2}$.
\end{thm}
The condition $e_C(P,Q;n)\neq 0$ can be easily understood by thinking about periodic random walks, for example the simple walk 
\begin{equation*}
    \chi_\mathcal S(x_1,\ldots,x_d)=x_1+\overline{x_1}+\cdots +x_d+\overline{x_d}, 
\end{equation*}
see our Examples~\ref{ex:ex_Coxeter_A} and \ref{ex:ex_Coxeter_B}, which can only return to its starting point in an even number of steps. Theorem~\ref{thm:DW_formula_exponent} was actually first proved under the following stronger aperiodicity property:
\begin{enumerate}[label={\rm (H\arabic*)},ref={\rm (H\arabic*)}]
\setcounter{enumi}{1}
\item\label{it:hypothesis_aperiodic} For any two points $P,Q$ in the domain $C$, the gcd of the set $\{n\in\mathbb N : e_C(P,Q;n)\neq 0\}$ is $1$.
\end{enumerate}
However, it was later proved \cite{DeWa-19} that the same asymptotic results hold only under \ref{it:hypothesis_irreducible}.

\subsection{Four random walks}
To be complete and to interpret the above Theorem~\ref{thm:DW_formula_exponent}, let us define four random walks that are naturally associated with any walk model $\mathcal S$. A similar discussion is proposed in dimension $2$ in \cite[Sec.~2.3]{BoRaSa-14}. First, the main random walk associated with the step set $\mathcal S$ and weight $w(s)$ is denoted by $(W_n)_{n\geq 0}$. By definition, for all $s\in\mathcal S$ and $n\geq 0$, $\mathbb P(W_{n+1}-W_n=s)=w(s)$. With $\chi_\mathcal S$ as in \eqref{eq:inventory}, it has a (possibly non-zero)\ drift given by $\nabla \chi_\mathcal S(\boldsymbol{1})=\sum_{s\in\mathcal S} w(s)s$ and a covariance matrix given by 
\begin{equation*}
    \left(\frac{\partial^2 \chi_\mathcal S}{\partial x_i\partial x_j}(\boldsymbol{1})\right)_{1\leq i,j\leq d}.
\end{equation*}

The second random walk model appears when the drift of the model is removed; it is the Cram\'er transformation, as mentioned in the introduction. It is denoted by $(X_n)_{n\geq0}$, has jumps in $\mathcal S$ and transitions 
\begin{equation}
    \label{eq:formula_new_weights}
    \mathbb P(X_{n+1}-X_n=s)=\frac{w(s)\boldsymbol{x_0}^{\boldsymbol{s}}}{\chi_\mathcal S(\boldsymbol{x_0})},
\end{equation}
with the multi-index notation $\boldsymbol{s}=(s_1,\ldots,s_d)\in\mathbb Z^d$ and $\boldsymbol{x_0}$ solving the system~\eqref{eq:chi_critical_point}. It has zero drift and covariance matrix given by
\begin{equation*}
\left(\frac{\frac{\partial^2 \chi_\mathcal S}{\partial x_i\partial x_j}(\boldsymbol{x_0})}{\chi_\mathcal S(\boldsymbol{x_0})}\right)_{1\leq i,j\leq d}.
\end{equation*}
If the initial random walk $(W_n)_{n\geq0}$ has zero drift, then $\boldsymbol{x_0}=\boldsymbol{1}$ and the first step becomes unnecessary: $X_n=W_n$.

The third model, denoted $(Y_n)_{n\geq0}$, is a normalised version of $(X_n)_{n\geq0}=(X_n^1,\ldots,X_n^d)_{n\geq0}$ so as to have a covariance matrix with unit diagonal coefficients; it is defined by
\begin{equation*}
    Y_n=\left(\frac{X_n^1}{\sqrt{\mathbb E\bigl((X_n^1)^2\bigr)}},\ldots,\frac{X_n^d}{\sqrt{\mathbb E\bigl((X_n^d)^2\bigr)}}\right).
\end{equation*}
It has zero drift and, by construction, a covariance matrix given by $\Delta$ in \eqref{eq:expression_covariance_matrix}.

Finally, the random walk $(Z_n)_{n\geq0} = (\Delta^{-\frac{1}{2}}Y_n)_{n\geq 0}$ has zero drift and an identity covariance matrix. Unlike the random walks $(W_n)_{n\geq 0}$, $(X_n)_{n\geq 0}$ and $(Y_n)_{n\geq 0}$, which all evolve in the orthant $\mathbb R_+^d$, the domain of definition of $(Z_n)_{n\geq0}$ is  $\Delta^{-\frac{1}{2}}\mathbb R_+^d$ as in \eqref{eq:def_new_domain}.

\section{A surjective morphism from the combinatorial group to a reflection group}
\label{sec:surjective}

\subsection{Preliminary results}

Our first result (Proposition~\ref{prop:reformulation_cov_matrix})\ is a reformulation of the covariance matrix given by \eqref{eq:expression_covariance_matrix} in terms of the cosine of certain angles. Our result is general and holds with general matrices; however, in our application we will systematically take the covariance matrix $\Delta$ as in \eqref{eq:expression_covariance_matrix}.

In the whole paper, we denote the canonical basis of $\R^d$ by $(e_i)_{1\leq i\leq d}$ and $\langle\cdot,\cdot\rangle$ will stand to the scalar product on $\R^d$.
We consider a matrix $\Delta = \bigl(a_{i,j}\bigr)_{1\leq i,j\leq d}$ which is symmetric, positive definite and has diagonal coefficients equal to $1$. The matrix $\Delta$ is diagonalisable in an orthonormal basis. We denote by \( P \in O(d) \) and \( D = \text{diag}(\alpha_1, \dots, \alpha_d) \), where for all $i$, $\alpha_i > 0$, the matrices such that \( \Delta = PDP^{-1} \). We can then define, for any \( s \in \R \), the matrix 
\begin{equation}
    \label{eq:def_power_Delta}
    \Delta^s = P \, \text{diag}(\alpha_1^s, \dots, \alpha_d^s) P^{-1}.
\end{equation}
In particular, the matrix $\Delta^{-\frac{1}{2}}$ appearing in Theorem~\ref{thm:DW_formula_exponent} is computed thanks to \eqref{eq:def_power_Delta} at $s=-\frac{1}{2}$.

The orthant $\mathbb R^d_+$ is bounded by the hyperplanes  
\begin{equation}
\label{eq:def_Gi}
G_i = \text{Span}\{e_1, \dots, e_{i-1}, e_{i+1}, \dots, e_d\},\quad i \in \{1, \dots, d\},
\end{equation}
and thus $T$ in \eqref{eq:def_T} is bounded by the hyperplanes \(H_i = \Delta^{-\frac{1}{2}} G_i \). 
We set
\begin{equation}
    \label{eq:def_u_i}
    u_i = \Delta^{\frac{1}{2}} e_i,
\end{equation}
A first observation is:
\begin{lem}\label{H_orth}
Let $u_i$ be defined in \eqref{eq:def_u_i}. It holds that 
\begin{equation}
    \label{eq:def_H_i}
    H_i = \langle u_i \rangle^\perp.
\end{equation}
\end{lem}
\begin{proof} Indeed, using that $\Delta^{-\frac{1}{2}}$ is a symmetric matrix, we have
\begin{align*}
    u \in \langle H_i\rangle^\perp &\Leftrightarrow \forall v \in H_i, \, \langle v,u \rangle =0 \,
    \Leftrightarrow \forall v \in \Delta^{-\frac{1}{2}} G_i, \, \langle v,u \rangle =0 \\
    &\Leftrightarrow \forall j\neq i , \, \langle \Delta^{-\frac{1}{2}} e_j , u \rangle =0 \,
    \Leftrightarrow \forall j\neq i , \, \langle e_j , \Delta^{-\frac{1}{2}} u \rangle =0\\
    &\Leftrightarrow \Delta^{-\frac{1}{2}} u = \lambda e_i \, , \, \, \lambda \in \R \,
    \Leftrightarrow u = \lambda \Delta^{\frac{1}{2}} e_i = \lambda u_i\, , \, \, \lambda \in \R.\qedhere
\end{align*}
\end{proof}

\begin{prop}
\label{prop:reformulation_cov_matrix}
Let $u_i$ be defined in \eqref{eq:def_u_i}. For all $i$, $\| u_i \| = 1$. Moreover, for all $i,j$, the angle between the vectors \( u_i \) and \( u_j \) is \( \alpha_{i,j} := \arccos a_{i,j} \in(0,\pi)  \).
\end{prop}
	
	\begin{proof}
To prove the first point, notice that $\Delta^{\frac{1}{2}}$ is a symmetric matrix, and thus
\begin{equation*}
    \|u_i \|^2 = \langle \Delta^{\frac{1}{2}} e_i, \Delta^{\frac{1}{2}} e_i \rangle = \langle e_i, \Delta e_i \rangle = e_i^\intercal \Delta e_i = a_{i,i} = 1 .
\end{equation*}
To show the second assertion, note that, on the one hand, \( \langle u_i, u_j \rangle = \cos(\alpha_{i,j}) \|u_i\| \|u_j\| = \cos(\alpha_{i,j}) \), and on the other hand, \( \langle u_i, u_j \rangle = \langle \Delta^{\frac{1}{2}} e_i, \Delta^{\frac{1}{2}} e_j \rangle = e_i^\intercal \Delta e_j = a_{i,j} \). Thus, \( \cos(\alpha_{i,j}) = a_{i,j} \).
	\end{proof}

We now compute the angles between any two hyperplanes defining \(T = \Delta^{-\frac{1}{2}} \mathbb{R}_+^d\). We define the angle \(\widehat{H_i H_j}\) between two hyperplanes \(H_i\) and \(H_j\) bounding $T$ by the interior angle within the orthant \(T\), as shown in Figure~\ref{fig:angle_Hi_Hj}. As the following result proves, the interior angles of $T= \Delta^{-\frac{1}{2}} \mathbb{R}_+^d$ can be read directly from the matrix $\Delta$.

	\begin{center}
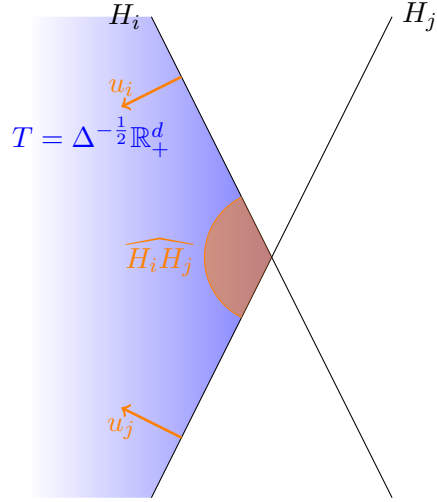
		\begin{tikzpicture}[scale=0.8]
		    \shade[left color=white, right color=blue, opacity=0.5] (-4,4)--(-2,4)--(0,0)--(-2,-4)--(-4,-4)--(-4,4);
		    \draw[blue] (-3,2) node{$T=\Delta^{-\frac{1}{2}} \mathbb{R}_+^d$};
			\draw (2,4) node[right]{$H_j$}--(-2,-4);
			\draw (-2,4) node[left]{$H_i$}--(2,-4);
			\draw[->, line width=1pt, color=orange] (-1.5,3)--(-2.5,2.5) node[above]{$u_i$};
			\draw[orange] (-1.12,0) node[left]{$\widehat{H_i H_j}$};
			\draw[->, line width=1pt, color=orange] (-1.5,-3)--(-2.5,-2.5) node[below]{$u_j$};
			\draw[orange] (-0.5,1) arc (116.57:243.43:1.12);
		    \fill[orange, opacity=0.5] (0,0)--(-0.5,1) arc (116.57:243.43:1.12)--cycle;
		\end{tikzpicture}
		\captionof{figure}{Measure of the angle $\widehat{H_i H_j}$ between two hyperplanes $H_i$ and $H_j$ in $T$, in terms of the angle $\widehat{\left(u_i,u_j\right)}$ between the normal vectors $u_i$ and $u_j$ in \eqref{eq:def_u_i}. The relationship between the angle $\widehat{H_i H_j}$ and $\widehat{\left(u_i,u_j\right)}$ is determined by the orientation of the normal vectors $u_i$ and $u_j$ being inward: $\widehat{\left(u_i,u_j\right)}=\pi - \widehat{H_i H_j}$.}
		 \label{fig:angle_Hi_Hj}
	\end{center}

\begin{prop}
\label{prop:angleHiHj}
It holds that 
 \begin{equation}
 \label{angle}
\widehat{H_i H_j}=-\arccos a_{i,j}.
\end{equation}
\end{prop}
\begin{proof}
Fix \( i \in \{1, \dots, d\} \). We prove that \( u_i \) is an inward,  normal unit vector with respect to $T$. We have already shown that $u_i$ is a normal unit vector, see Lemma~\ref{H_orth}. To prove the inward property, we take any point on $v \in H_i \cap \partial T$ (but not in $H_j$ for $j \not=i$)  and show  that for small $t>0$,  \( v + t u_i \in T  \).
Take for instance $v = \Delta^{-\frac{1}{2}} w$, with
\[w =  e_1 + \dots +  e_{i-1} +  e_{i+1} + \dots +  e_d .
\]
We have $w \in G_i \cap \partial \R_+^d$ and $w \not\in G_j$ if $j \not=i$, remembering that \( G_i \) is defined in \eqref{eq:def_Gi}. Using~\eqref{eq:def_u_i}, we can write
$$
    v+tu_i
    = \Delta^{-\frac{1}{2}}( w + t\Delta e_i).
$$
Denoting by \( \mu_1, \dots, \mu_d \) the coordinates of  $w + t  \Delta e_i$ in the canonical basis, we get that, since  $\Delta=(a_{i,j})_{1 \leq i,j \leq d}$:  
\[
\forall j \in \{1, \dots, d\} \setminus \{i\}, \quad \mu_j = 1 + t a_{j,i} \quad \text{and} \quad \mu_i = t a_{i,i} = t.
\]
All these coordinates are positive for small $t>0$ and thus $w+t\Delta e_i \in \R_+^d$, which proves that   $v+tu_i\in T = \Delta^{-\frac{1}{2}} \R_d^+$. This shows that  \( u_i \) is inward, which clearly implies that for all \( i, j \in \{1, \dots, d\} \),   
\[
\widehat{\left( u_i, u_j \right)} = \pi - \widehat{H_i H_j}.
\]
By applying Proposition ~\ref{prop:reformulation_cov_matrix}, 
Equation~\eqref{angle} follows. 
\end{proof}

Let \( \mathcal S \) be a set of directions and \( \chi_\mathcal S \) the associated polynomial \eqref{eq:inventory}. We denote by \( G \) the group \eqref{eq:group_G_def} generated by the transformations \eqref{eq:expression_generators}, which we reformulate as
\( \varphi_i = \langle x_1, \dots, x_{i-1}, \sigma_i, x_{i+1}, \dots, x_d \rangle \), with \( \sigma_i = \frac{1}{x_i} \frac{C_i}{A_i} \). From now on, $\Delta=(a_{i,j})_{1 \leq i,j \leq d}$ is the covariance matrix introduced in Theorem \ref{thm:DW_formula_exponent}.

\subsection{Main result}

The main result in this section is a reformulation of the group $G$ as the preimage by a morphism of a reflection group. To state this result, we recall our notation $H_i = \langle u_i \rangle^\perp$ and introduce $r_i$, the orthogonal reflection with respect to $H_i$. Finally, we introduce the reflection group
\begin{equation}
    \label{eq:group_H_def}
    H = \langle r_1,\ldots,r_d\rangle.
\end{equation}
By definition, this is the group generated by the reflections with respect to the sides of the linear transformation of the orthant \( \Delta^{-\frac{1}{2}} \R_+^d \). 

We need to define two more quantities. First, we introduce the application
	\begin{equation}
	    \label{eq:Jacobian_application}	    	J :\left|\begin{array}{ccc}
		  G & \to & GL_d(\R) \\
		 g & \mapsto & \Jac_{\boldsymbol{x_0}} g 
	\end{array}\right.
	\end{equation}
	where \( \boldsymbol{x_0} \) is the unique minimum of \( \chi_\mathcal S \), see \eqref{eq:chi_critical_point}. 
Secondly, we define a symmetric bilinear form on $\mathbb R^d$ by setting, for all $h, k \in \R^d$,
\begin{equation}
\label{eq:scalar_product_def}
    [h, k] := \text{D}^2 \chi_\mathcal S (\boldsymbol{x_0})(h, k) = \sum\limits_{i,j=1}^d \frac{\partial^2 \chi_\mathcal S}{\partial x_i \partial x_j}(\boldsymbol{x_0}) h_i k_j
\end{equation}
and we set \( f_i = \frac{e_i}{\sqrt{[e_i, e_i]}} \).
The function in \eqref{eq:scalar_product_def} is the quadratic form associated with the covariance matrix of the random walk $(X_n)_{n\geq0}$ introduced in Section~\ref{sec:preliminaries}. Notice that the Jacobian of some elements of the group $G$ was already used in \cite[Sec.~3]{BMMi-10} (for two-dimensional models)\ and in \cite{DuHoWa-16,KaWa-17} (for three-dimensional models), in order to prove that in some cases, the group $G$ is infinite. As we will see later, the Jacobian application \eqref{eq:Jacobian_application} can be used to obtain various and new results about the group $G$. In particular, it works in any dimension and applies to the comparison of the groups $G$ in \eqref{eq:group_G_def} and $H$ in \eqref{eq:group_H_def}.

\begin{thm}
\label{thm:prop_scalar_product}
The following assertions hold true:
\begin{enumerate}[label={\rm(\roman*)},ref={\rm(\roman*)}]
\item The map $J$ is a morphism.
\item\label{it1:scalar} \( [\cdot,\cdot] \) is a scalar product.
			\item\label{it2:scalar}\( [\cdot,\cdot] \) is invariant under \( \Ima J \).
			
			\item\label{it3:scalar}The map
			\begin{equation}
			    \label{eq:def_Phi}
			   	\Phi :\left|\begin{array}{ccc}
				  \R^d & \to & \R^d\\
				 f_i & \mapsto & u_i 
			\end{array}\right.
			\end{equation}
			is an isometry from \( (\R^d, [\cdot,\cdot]) \) to \( (\R^d, \langle\cdot,\cdot\rangle) \).
			\item\label{it4:scalar}For all \( i \), \( S_i = J \varphi_i \) is the orthogonal reflection with respect to \( e_i \) in \( (\R^d, [\cdot,\cdot]) \).
		\end{enumerate}
	\end{thm}
	
	\begin{proof}
Let us first show that \( J \) is a morphism. A first well-known observation is that since \( \boldsymbol{x_0} \) is the minimum of \( \chi_{\mathcal S} \), $\boldsymbol{x_0}$ is a fixed point of all $g \in G$. 
It suffices to check this fact for all generators $\varphi_i$ of $G$. Since \( \chi_\mathcal S = \overline{x_i} C_i + B_i + x_i A_i \), we have, if we note $\boldsymbol{x_0}=(x_0^1 \ldots,x_0^d)$,
\[
0=\frac{\partial \chi_\mathcal S}{\partial x_i}\left(\boldsymbol{x_0}\right) = \frac{-1}{\bigl(x_0^i\bigr)^2} C_i\left(\boldsymbol{x_0}\right) + A_i\left(\boldsymbol{x_0}\right).
\]
Thus, 
\[
\varphi_i\left(\boldsymbol{x_0}\right) = \left[x_0^1, \dots, x_0^{i-1}, \frac{1}{x_0^i} \frac{C_i\left(\boldsymbol{x_0}\right)}{A_i\left(\boldsymbol{x_0}\right)}, x_0^{i+1}, \dots, x_0^d\right] = \boldsymbol{x_0}.
\]
For all $g,g' \in G$, we thus get 
\[
J(gg')=\text{Jac}_{\boldsymbol{x_0}} (g \circ g') = \text{Jac}_{g'(\boldsymbol{x_0})} g \, \text{Jac}_{\boldsymbol{x_0}} g' = \text{Jac}_{\boldsymbol{x_0}} g \, \text{Jac}_{\boldsymbol{x_0}} g' = J(g) J(g').
\]
We now prove \ref{it1:scalar}. Obviously using \eqref{eq:expression_covariance_matrix} one has for $h \in \mathbb R^d$ 
\begin{equation*}
    [h, h] = \text{D}^2 \chi_{\mathcal S} (\boldsymbol{x_0})(h, h) = \sum\limits_{i,j=1}^d \frac{\partial^2 \chi_{\mathcal S}}{\partial x_i \partial x_j}(\boldsymbol{x_0}) h_i h_j >0.
\end{equation*}
The positivity follows from the positive definiteness of the covariance matrix of the random walk $(X_n)_{n\geq 0}$, which in turn follows from \ref{it:hypothesis_irreducible}.

We now provide the proof of \ref{it2:scalar}, by showing that
\begin{equation}
    \label{eq:to_prove_JacJac}
    [J(g)h, J(g)k]=[\text{Jac}_{\boldsymbol{x_0}} g(h), \text{Jac}_{\boldsymbol{x_0}} g(k)] = [h, k]
\end{equation}
for all $g\in G$ and all \( h, k \in \R^d \). We have
			\begin{align*}
		[J(g)h, J(g)k]  & = \text{D}^2 \chi_{\mathcal S} (\boldsymbol{x_0})(\text{Jac}_{\boldsymbol{x_0}} g(h), \text{Jac}_{\boldsymbol{x_0}} g(k))\\ 
				& = \sum\limits_{i,j=1}^d \frac{\partial^2 \chi_{\mathcal S}}{\partial x_i \partial x_j}(\boldsymbol{x_0}) (\text{Jac}_{\boldsymbol{x_0}} g(h))_i (\text{Jac}_{\boldsymbol{x_0}} g(k))_j.	
			\end{align*}
By differentiating twice the equality \( \chi_{\mathcal S} \circ g = \chi_{\mathcal S} \) at $\boldsymbol{x_0}$, we obtain
			\[
			\text{D}^2\chi_{\mathcal S}(g(\boldsymbol{x_0})) (J(g)h, J(g)k)   + \text{D}\chi_{\mathcal S}(g(\boldsymbol{x_0})) \circ \text{D}^2g(\boldsymbol{x_0})(h,k) =\text{D}^2\chi_{\mathcal S}(g(\boldsymbol{x_0})) (h, k).		\]
			Since $g(\boldsymbol{x_0})=\boldsymbol{x_0}$ is a critical point of $\chi_\mathcal S$, the last term of the left-hand side vanishes and we get 
			$$		\text{D}^2\chi_{\mathcal S}(g(\boldsymbol{x_0})) (J(g)h, J(g)k) = 	\text{D}^2\chi_{\mathcal S}(g(\boldsymbol{x_0})) (h, k),$$ i.e.,
			$[J(g)h,J(g)k]= [h,k]$.
		This proves Equation \eqref{eq:to_prove_JacJac}.
			
We move to the proof of \ref{it3:scalar}. We show that for all \( i,j \in \{ 1, \dots, d \} \),
$ \langle u_i , u_j \rangle = [f_i ,f_j]$, with $u_k$ defined in \eqref{eq:def_u_i} and \( f_k = \frac{e_k}{\sqrt{[e_k, e_k]}} \). Indeed, we have
\begin{equation*}
    [f_i, f_j] = \left[ \frac{e_i}{\sqrt{[e_i, e_i]}}, \frac{e_j}{\sqrt{[e_j, e_j]}} \right] = \frac{\frac{\partial^2 \chi_{\mathcal S}}{\partial x_i \partial x_j}(\boldsymbol{x_0})}{\sqrt{\frac{\partial^2 \chi_{\mathcal S}}{\partial x_i^2}(\boldsymbol{x_0}) \cdot\frac{\partial^2 \chi_{\mathcal S}}{\partial x_j^2}(\boldsymbol{x_0})}} = a_{i,j} = \langle u_i , u_j \rangle,
\end{equation*}
where we recall that $\Delta=(a_{i,j})_{1 \leq i,j\leq d}$ is the covariance matrix introduced in Theorem \ref{thm:DW_formula_exponent} and where the last equality comes from Proposition~\ref{prop:reformulation_cov_matrix}.

Finally we show \ref{it4:scalar}.  First notice that since $\varphi_i^2=\Id$, it holds that  $S_i^2=(J\varphi_i)^2=\Id$. We compute the following matrix form of $S_i$, where we denote by $\varphi_i^j$ the $j$-th coordinate function of \( \varphi_i = \bigl( x_1, \dots, x_{i-1}, \sigma_i, x_{i+1}, \dots, x_d \bigr) \):
\begin{equation}
\label{eq:expression_S_i}
	S_i = \left(\begin{array}{ccccccc}
	1 & 0 & \cdots & \cdots & \cdots & \cdots & 0\\
	0 & \ddots & \ddots &  \textcolor{white}{\frac{\partial \varphi_i^{2}}{\partial x_4}\left(\boldsymbol{x_0}\right)} & \textcolor{white}{\frac{\partial \varphi_i^{3}}{\partial x_5}\left(\boldsymbol{x_0}\right)} &  \textcolor{white}{\frac{\partial \varphi_i^{2}}{\partial x_{n-1}}\left(\boldsymbol{x_0}\right)} & \vdots\\
	\vdots & \vdots & 1 & 0 & \cdots & \cdots & 0 \\
	\frac{\partial \varphi_i^i}{\partial x_1}\left(\boldsymbol{x_0}\right) & \frac{\partial \varphi_i^i}{\partial x_2}\left(\boldsymbol{x_0}\right) & \cdots & -1 & \cdots & \cdots & \frac{\partial \varphi_i^i}{\partial x_n}\left(\boldsymbol{x_0}\right)\\
	0 & 0 & \cdots & 0 & 1& 0& \vdots\\
	\vdots & \vdots & \textcolor{white}{\frac{\partial \varphi_i^{n-1}}{\partial x_3}\left(\boldsymbol{x_0}\right)} && \ddots & \ddots& \vdots\\
	0 & 0 & \cdots & \cdots & \cdots & 0 & 1\\
	\end{array}\right).
\end{equation}
To prove \eqref{eq:expression_S_i}, we only need to justify the calculations leading to the $i$-th row. Since $\sigma_i=\frac{1}{x_i}\frac{C_i}{A_i}$ by \eqref{eq:expression_generators}, we have $\frac{\partial \varphi_i^i}{\partial x_i}(\boldsymbol{x_0})=\frac{\partial \sigma_i}{\partial x_i}(\boldsymbol{x_0})= -\frac{1}{x_i^2}\frac{C_i}{A_i}$. On the other hand, since $\boldsymbol{x_0}$ is the minimum of $\chi_{\mathcal S}$, we have $\frac{\partial \chi_{\mathcal S}}{\partial x_i}(\boldsymbol{x_0})= -\frac{1}{x_i^2}C_i + A_i = 0$, so that $\frac{\partial \varphi_i^i}{\partial x_i}(\boldsymbol{x_0})=-1$.

The matrix $S_i$ in \eqref{eq:expression_S_i} is diagonalisable, with eigenvalues $-1$ (simple) and $1$  (multiplicity $d-1$). Moreover, from  \eqref{eq:expression_S_i}, $e_i$ is clearly a eigenvector associated to $-1$.

We denote by \( E_i = \{ x \in \R^d \mid S_i(x) = x \} \) the eigenspace of \( \varphi_i \) associated with eigenvalue \( 1 \) and show that \( E_i = \langle e_i \rangle^{\perp_{[\cdot,\cdot]}} \). For \( x \in E_i \), since \( [\cdot,\cdot] \) is invariant under \( \Ima J \) by \ref{it2:scalar}, 
\begin{equation*}
   [x, e_i] = [S_i(x), e_i] = [S_i^2(x), S_i(e_i)] = [x, -e_i].    
\end{equation*}
Hence, \( [x, e_i] = 0 \) and \( E_i \subset \langle e_i \rangle^{\perp_{[\cdot,\cdot]}} \). But since $E_i$ has dimension $d-1$, we get \( E_i = \langle e_i \rangle^{\perp_{[\cdot,\cdot]}} \). Thus, \( S_i \) is indeed the orthogonal reflection with respect to \( E_i \) in \( (\R^d, [\cdot,\cdot]) \).
	\end{proof}
	
We will now state two important consequences of Theorem~\ref{thm:prop_scalar_product}, all of which concern the groups $G$ and $H$. In the first corollary below, we show that the image of $G$ by $J$ in \eqref{eq:Jacobian_application} is isomorphic to the reflection group $H$. 
Before giving the result, we first note that $\Phi$ in \eqref{eq:def_Phi} induces an isomorphism $\widetilde{\Phi} $ between the orthogonal groups $O(\R^d,[\cdot,\cdot])$ and $O(\R^d,\langle \cdot , \cdot \rangle)$.  More precisely, for all $s \in O(\R^d,[\cdot,\cdot])$, define 
$$\widetilde{\Phi}(s) = \Phi \circ s \circ \Phi^{-1},$$ which belongs to $ O(\R^d,[\cdot,\cdot])$ since $\Phi$ is an isometry. Note that by construction we have $\widetilde{\Phi}(S_i)= r_i$ for all $i$. As an immediate consequence we get:
	\begin{cor}
	\label{cor:imJH}
	 The restriction of $\widetilde{\Phi}$ to $\Ima J$ is an isomorphism between the two reflection groups $\Ima J$ and $H$ such that for all $i$, $\widetilde{\Phi}(S_i)=r_i$. In particular, $\widetilde{\Phi} \circ J : G \to H$ is a surjective morphism that sends the set of generators $(\varphi_1, \ldots, \varphi_d)$ of $G$ to the set of generators $(r_1,\ldots,r_d)$ of $H$. 
	\end{cor}

While Corollary~\ref{cor:imJH} shows a strong connection between the groups $G$ and $H$, it is important to note that in general these two groups do not coincide. See Example~\ref{ex:ex4} for a concrete example where the group $H$ is of order $8$ while $G$ is infinite. However, we obtain the following consequence:	
	\begin{cor} 
	\label{GinHfin}
		If \( G \) is finite, then \( H \) is finite.
	\end{cor}
In Propositions~\ref{GHisom}, \ref{presentation}  and \ref{presentation2}, we give sufficient conditions for the groups $G$ and $H$ to be isomorphic. As we will explain in Section~\ref{sec:app_ex}, a possible application of Corollary~\ref{GinHfin} is that if $H$ is infinite, then necessarily $G$ should be infinite. This may be of practical interest, since reflection groups are well understood and more classical than the combinatorial group $G$.

\subsection{Illustration in dimension two}
\label{subsec:illustration_2}

The case of dimension two is the one that has attracted the most attention in the literature, see e.g. \cite{Ma-71,FaIaMa-17,BMMi-10,BoRaSa-14}. Our results do not bring any new progress in this case, but it is interesting to compute the new domain $T$ in \eqref{eq:def_new_domain} and see how it depends on the parameters. 
The covariance matrix \eqref{eq:expression_covariance_matrix} takes the form 
\begin{equation*}
    \Delta = \left(\begin{array}{cc}
    1 & a\\ a&1 \end{array}\right),\quad \text{with } 
     a=a_{1,2}=\frac{\frac{\partial^2 \chi_\mathcal S}{\partial x_1\partial x_2}(\boldsymbol{x_0})}{\sqrt{\frac{\partial^2 \chi_\mathcal S}{\partial x_1^2}(\boldsymbol{x_0})\cdot \frac{\partial^2 \chi_\mathcal S}{\partial x_2^2}(\boldsymbol{x_0})}}.
\end{equation*}
Diagonalizing the matrix $\Delta$ and setting $a=-\cos \alpha$, we easily find
\begin{equation*}
    \Delta^{\frac{1}{2}} =  \left(\begin{array}{rr}
    \cos\bigl(\frac{\pi}{4}-\frac{\alpha}{2}\bigr) & -\sin\bigl(\frac{\pi}{4}-\frac{\alpha}{2}\bigr)\medskip\\ -\sin\bigl(\frac{\pi}{4}-\frac{\alpha}{2}\bigr)&\cos\bigl(\frac{\pi}{4}-\frac{\alpha}{2}\bigr) \end{array}\right).
\end{equation*}
The vectors $u_1$ and $u_2$ in \eqref{eq:def_u_i} are equal to
\begin{equation*}
    u_1 = \left(\begin{array}{r}
    \cos\bigl(\frac{\pi}{4}-\frac{\alpha}{2}\bigr) \\ -\sin\bigl(\frac{\pi}{4}-\frac{\alpha}{2}\bigr) \end{array}\right) \quad \text{and} \quad u_2 = \left(\begin{array}{r}
    -\sin\bigl(\frac{\pi}{4}-\frac{\alpha}{2}\bigr) \\ \cos\bigl(\frac{\pi}{4}-\frac{\alpha}{2}\bigr) \end{array}\right),
\end{equation*}
see Figure~\ref{fig:dim2_vectors}. The hyperplanes $H_i = \langle u_i \rangle^\perp$ in \eqref{eq:def_H_i} are thus given by
\begin{equation*}
    H_1 = \left(\begin{array}{r}
    \sin\bigl(\frac{\pi}{4}-\frac{\alpha}{2}\bigr) \\ \cos\bigl(\frac{\pi}{4}-\frac{\alpha}{2}\bigr) \end{array}\right) \quad \text{and} \quad H_2 = \left(\begin{array}{r}
    \cos\bigl(\frac{\pi}{4}-\frac{\alpha}{2}\bigr) \\ \sin\bigl(\frac{\pi}{4}-\frac{\alpha}{2}\bigr) \end{array}\right).
\end{equation*}
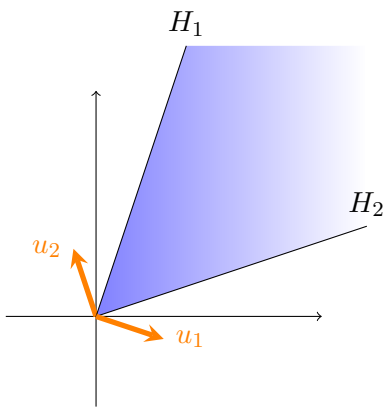
\begin{figure}
    \begin{tikzpicture}[scale=0.60]
        \shade[left color=blue, right color=white, opacity=0.5] (0,0)--(2,6)--(6,6)--(6,2)--(0,0);
        \draw[->] (-2,0)--(5,0);
        \draw[->] (0,-2)--(0,5);
        \draw (0,0)--(6,2) node[above]{$H_2$};
        \draw[->, > = stealth, line width=2pt, color=orange] (0,0)--(-0.5,1.5) node[left]{$u_2$};
        \draw (0,0)--(2,6) node[above]{$H_1$};
        \draw[->, > = stealth, line width=2pt, color=orange] (0,0)--(1.5,-0.5) node[right]{$u_1$};
    \end{tikzpicture}
    \caption{In dimension $2$, the vectors $u_i$ and the hyperplanes $H_i$ are easily computed in terms of the correlation factor $a$. The chamber becomes the domain between the two hyperplanes $H_1$ and $H_2$, which is a wedge of opening $\arccos -a$, as already noticed in \cite[Ex.~2]{DeWa-15}. See also Figure~\ref{fig:action_Delta}.}
    \label{fig:dim2_vectors}
\end{figure}

\section{Applications}
\label{sec:app}

In our opinion, the main interest of our results (Theorem~\ref{thm:prop_scalar_product} and Corollary~\ref{cor:imJH})\ is that they clarify a lot about the connections between the combinatorial group $G$ in \eqref{eq:group_G_def}, the reflection group $H$ in \eqref{eq:group_H_def} and Coxeter groups in general. However, our results also have concrete applications; in particular, they provide some tools to determine whether $G$ is finite or not in several situations.

\subsection{Infinite group criterion in any dimension}
\label{sec:app_ex}

In dimension two and three, the classification of the models with respect to the (in)finiteness of the group $G$ is complete in the case of unweighted models $\mathcal S$ (unweighted means that all non-zero weights $w(s)$ are equal; in other words, the walk jumps uniformly to any element of the step set $\mathcal S$); see \cite{BMMi-10} for the case of dimension two, and \cite{BoBMKaMe-16,DuHoWa-16,BaKaYa-16,KaWa-17} for the case of dimension three. While one can observe by direct computation that a given model admits a finite group (computing all elements of the group, see e.g.\ \cite{BMMi-10,BoBMKaMe-16}), proving that the group $G$ is infinite is more delicate. Let us recall some possible strategies:
\begin{itemize} 
  \item A first observation is that if $H$ is infinite, then $G$ is also infinite. This is obviously a direct application of Corollary~\ref{GinHfin}, but our result is not really needed in the present situation. Indeed, to prove that $G$ is infinite, it is sufficient to show that for some $g\in G$ the matrix $J(g)$ in \eqref{eq:Jacobian_application} is of infinite order. This is the strategy proposed in \cite{BMMi-10} for unweighted models in dimension two. In practice, it is shown that the eigenvalues of $J(g)$ have norm $1$ and that their order on the unit circle $\mathbb{S}^1$ is infinite. 
  \item The above method may fail; typically, in our notation, if $H$ is finite, then $J(g)$ has finite order for any $g\in G$; see Example~\ref{ex:ex4}. In such cases, as we now explain following the authors of \cite{DuHoWa-16}, the argument can be adapted. In fact, it is not necessary to calculate the order of the Jacobian matrix at $\boldsymbol{x_0}$ in \eqref{eq:chi_critical_point}. While $\boldsymbol{x_0}$ is the unique fixed point common to all elements of $G$, any given $g\in G$ admits many more fixed points. For example, denoting $\varphi_1$ and $\varphi_2$ as two generators of the group $G$ as in \eqref{eq:group_G_def}, there exists a fixed point $x_z= (x,y,z)$ of $\varphi_1$ and $\varphi_2$, and thus of the composition $\varphi_1 \varphi_2$, for any value of $z\in\mathbb R^{d-2}$. In the method in \cite{DuHoWa-16}, the difficulty is to find a fixed point where $J_{x_z}(g)$ has an eigenvalue of norm other than $1$. However, if it can be found, then $J_{x_z}(g)$ is of infinite order, and thus the group $G$ is infinite. 
\end{itemize}
The above methods require long and case-by-case computations, with a number of models that grows with dimension (more than $11$ millions of unweighted in dimension three, see \cite{BoBMKaMe-16}). 

Our results allow the previous arguments to be greatly simplified in several situations. As already explained, to prove that $G$ is infinite, it is sufficient to prove that $H$ is infinite. The methods above actually consist of proving that $\Ima J$ is infinite, but working directly with $H$ can be very helpful:
\begin{itemize}
    \item First, note that given $g= \varphi_i \varphi_j$, the order of $J(g)$ is known as soon as $\boldsymbol{x_0}$ is computed. In fact, using our notation from Theorem~\ref{thm:prop_scalar_product}, $J(g)= S_iS_j$ (or $r_ir_j \in H$)\ is then the composition of two reflections and thus a rotation of the angle twice the angle between the hyperplanes $H_i$ and $H_j$ (see \eqref{eq:def_H_i}). Using Proposition~\ref{prop:angleHiHj}, its order can be read directly from the covariance matrix $\Delta$ in \eqref{eq:expression_covariance_matrix}: it is the order of $e^{2i\arccos a_{i,j}}$ in the circle $\mathbb S^1$, where $a_{i,j}$ is the coefficient $(i,j)$ of the covariance matrix $\Delta$. 
    \item If $H$ is finite, the method fails, but we can apply the same ideas to a subgroup $G'$ of $G$ at a fixed point of $G'$.
    \item Another property of $H$ is that it is a reflection group and thus a Coxeter group (see Appendix~\ref{sec:coxeter} for some reminders about Coxeter groups), which are classified. If it is finite, it must belong to a short list of examples, and so to prove that $H$ is infinite it suffices to exclude $H$ from the list of examples. This allows us to deduce  Proposition~\ref{prop:app} below.
\end{itemize} 
In Proposition~\ref{GHisom} below we will further consider the case of a finite group $G$ and give a condition ensuring that $G$ and $H$ are isomorphic.

\subsection{A method to prove that the reflection group is infinite}

We prove the following:
\begin{prop}
\label{prop:app}
Let $(a_{i,j})_{1 \leq i,j \leq d}$ be the coefficients of the covariance matrix $\Delta$ in \eqref{eq:expression_covariance_matrix}. Then the group $H$ is infinite as soon as the following two conditions are satisfied:
\begin{itemize}
    \item $\Delta$ is not, up to a permutation of lines, a matrix diagonal by blocks with a block of size $2$; 
    \item there exist $i \not= j$ such that  $a_{i,j} \not\in \bigl\{0, \pm \frac{1}{2}, \pm \frac{\sqrt{2}}{2}, \pm \frac{\sqrt{3}}{2}, \pm \frac{\sqrt{5}-1}{4}, \pm \frac{ \sqrt{5}+1}{4} \bigr\}$, or equivalently such that $a_{i,j}$ has not the form $\cos(\frac{k \pi}{m})$ with $k \in \Z$ and 
$m\in \{1,2,3,4,5,6\}$.
\end{itemize}
\end{prop}

Let us recall that $H$ is the group spanned by  $r_1,\ldots,r_d$,  the reflections with respect to $H_1:= \langle u_1\rangle^\perp, \ldots, H_d:= \langle u_d\rangle^\perp$, where $u_i= \Delta^{\frac{1}{2}} e_i$, see \eqref{eq:def_u_i}, \eqref{eq:def_H_i} and \eqref{eq:group_H_def}. A technical difficulty is that the conditions of Proposition~\ref{prop:app} do not ensure that $r_1,\ldots,r_d$ is a Coxeter system (see Definition~\ref{def:Coxeter_system} in Appendix~\ref{sec:coxeter} for the concept of a Coxeter system), and so to derive conditions on $a_{i,j}$ we cannot directly use Proposition~\ref{possible_order}, which gives the list of Coxeter systems spanning finite Coxeter groups.

\begin{proof}[Proof of Proposition~\ref{prop:app}]
First, since $(u_1,\ldots,u_d)$ is a basis of $\R^d$, the Coxeter group $H$ has rank $d$. Suppose it is finite. We do not know if it is reducible or not, but assume that it has an irreducible factor (see Appendix~\ref{sec:coxeter}), which is a dihedral group of order $2k$. Consider the set of all roots of $H$, i.e.
\begin{equation}
\label{eq:def_roots}
   R:= \bigl\{ \pm h u_i | h \in H \bigr\}.
\end{equation}
Then, from the classification of finite Coxeter groups, there are exactly $2k$ roots (corresponding to $k$ hyperplanes) which belong to a plane and are orthogonal to the other roots. More precisely, the dihedral group factor is spanned by the $k$ reflections with respect to these hyperplanes.  Since $(u_1,\ldots,u_d)$ is a basis, it must contain exactly two of these $2k$ roots, which are then orthogonal to the other $u_i$. Up to a permutation of the $u_i$ we can assume that these roots are $u_1$ and $u_2$. By Proposition~\ref{prop:reformulation_cov_matrix}, the coefficient $(i,j)$ of $\Delta$ is $0$ if $i=1,2$ and $j\geq 3$, since $u_i \perp u_j$. The matrix $\Delta$ thus has a $2$-diagonal block, which is a contradiction. We deduce that if the assumptions of Proposition~\ref{prop:app} are fulfilled, then $H$ has no irreducible dihedral group factor.

Let now $\Lambda:= \cup_{w \in R} \langle w\rangle^\perp$, with $R$ as in \eqref{eq:def_roots}. Suppose there exist $i \not= j$ such that $\arccos a_{i,j}$ does not have the form $\frac{k\pi}{m}$ for some $k \in \Z$ and $m\in \{1,2,3,4,5,6\}$, and thus $e^{2 i\arccos a_{i,j}}$ has order $\al_{i,j}$ greater than $7$ in $\mathbb S^1$.  Since $r_i$ and $r_j$ are reflections, this means that  $r_i r_j$ is a rotation in the plane $P$ orthogonal to $P^\perp:= H_i \cap H_j$, with angle $-2\arccos a_{i,j}$ which is twice  the angle between $H_i$ and  $H_j$, see \eqref{angle}. The group $\{ (r_ir_j)^k | k\in \Z\}$ is the rotation group spanned by $e^{2 i\arccos a_{i,j}}$ in  $\mathbb S^1$ and since $\al_{i,j} \geq 7$, there exists $k$ such that $(r_ir_j)^k$ is a rotation of angle $\theta \in [0, \frac{2\pi}{7}]$. Introducing $H'_i:= (r_ir_j)^k(H_i)$, this means that $H_i,H_i' \in \Lambda$ are two hyperplanes of $\Lambda$ with angle $\theta$.
 In particular, every chamber of $K$ has two walls whose angle is less than $2 \pi/7$, which by Proposition~\ref{possible_order} implies that $H$ is infinite. 
\end{proof}

\begin{ex}\normalfont
\label{ex:ex1}
We will now illustrate how we can prove that $H$ (and thus $G$)\ is infinite, using simple calculations. Let $(x,y,z,w)$ be the standard Cartesian coordinates of $\R^4$. We 
consider the model whose inventory \eqref{eq:inventory} is given by 
$$\chi_\mathcal S(x,y,z,w)= xy +\overline{y} z +\overline{z}w + \overline{x}yz+\overline{y}z\overline{w}+xy\overline{z}+\overline{x}\overline{y}\overline{z}.$$ 
Here we show that $H$ is infinite by computing the covariance matrix.
By direct calculation (or using the fact that the drift of the model is zero), the fixed point is $\boldsymbol{x_0}=(1,1,1,1)$ and the covariance matrix is 
\[ \Delta = \left( \begin{array}{cccc} 
1 & \frac{1}{\sqrt{6}} & -\frac{1}{2\sqrt{6}} & 0 \\
 \frac{1}{\sqrt{6}} & 1 & -\frac{1}{6} & \frac{1}{2 \sqrt{3}} \\
-\frac{1}{2\sqrt{6}} & -\frac{1}{6} & 1 & - \frac{1}{ \sqrt{3}} \\
0 & \frac{1}{2 \sqrt{3}} & - \frac{1}{ \sqrt{3}} & 1 
\end{array}
\right). \]
Proposition~\ref{prop:app} immediately implies that $H$ is infinite, and therefore $G$ is also infinite. 
\end{ex}
The method presented in Example~\ref{ex:ex1} could be made systematic on a large class of models; we do not explore this line of research in this paper.

\subsection{A tool to determine \texorpdfstring{$\boldsymbol{H}$}{H} when \texorpdfstring{$\boldsymbol{G}$}{G} is known}
We first need the following observation.
\begin{lem}
\label{min_order}
Let $K$ be a rank $d$ reflection group. Then $\vert K\vert \geq 2^d$ with equality if and only if 
$H = \left(\frac{\mathbb Z}{2 \mathbb Z} \right)^d$. 
\end{lem}
\begin{proof}
It comes either from the classification of finite Coxeter groups or from the classical Matsumoto's theorem. 
Let $(s_1,\ldots,s_d)$ be a Coxeter system. Define 
$$X:\left| \begin{array}{ccc} \{0,1\}^d&  \to & K\\
(a_1,\ldots,a_d) & \mapsto  & s_1^{a_1} \cdots s_d^{a_d}
\end{array} \right.$$
Then by Matsumoto's theorem \cite{Mat-64}, $X$ is injective which implies that $\vert K\vert \geq 2^d$. Assume now that $|K| = 2^d$.
This means that 
$K=X( \{0,1\}^d )$. Let $i<j$, then $s_j s_i$ can be written as $s_j s_i= s_{i_1} \cdots s_{i_r}$ for some $1 \leq i_1 < \cdots < i_r \leq d$. Again by Matsumoto's theorem, we have $\{j,i\}=\{i_1, \ldots,i_r\}$. Since $i<j$, this implies that $r=2$, $i_1=i$ and $i_2=j$, and consequently that $s_j s_i = s_i s_j$. This proves that $H = \left(\frac{\mathbb Z}{2 \mathbb Z} \right)^d$. 
\end{proof}

We are ready to give the following application of Corollary~\ref{cor:imJH}, which gives a sufficient condition for the groups $G$ and $H$ to be isomorphic.
\begin{prop} 
\label{GHisom}
Assume that $G$ is finite and let 
\begin{equation*}
    N:=\min\bigl\{|K| \; \big| \; K \hbox{ is a normal subgroup of } G, \, K\not=\{\Id\} \bigr\}.
\end{equation*}
\begin{itemize}
    \item If $|G| < 2^d N$, then $G$ and $H$ are isomorphic. 
    \item If  $|G| = 2^d N$ and $G$ and $H$ are not isomorphic, then  $H = \left(\frac{\mathbb Z}{2 \mathbb Z} \right)^d$.
\end{itemize}
\end{prop}
\begin{proof} 
Assume first that $|G| < 2^d N$. By Corollary~\ref{cor:imJH}, $H$ is isomorphic to $G / \ker(\widetilde{\Phi} \circ J)$.   Since $H$ is a finite reflection group of rank $d$, we can deduce from Proposition~\ref{min_order} that $|H| \geq 2^d$. As a consequence, if  $\ker(\widetilde{\Phi} \circ J) \not= \{\Id\}$, then necessarily
$$|G| \geq |H|\; |\ker(\widetilde{\Phi} \circ J)| \geq 2^d N,$$
which contradicts our assumption. This implies that $\ker(\widetilde{\Phi} \circ J) = \{\Id\}$, and so the groups $G$ and $H$ are isomorphic. 

Now suppose that $|G| = 2^d N$ and that $G$ and $H$ are not isomorphic. We must have $|H|=2^d$ and we deduce that $H = \left(\frac{\mathbb Z}{2 \mathbb Z} \right)^d$. 
\end{proof}

We can apply Proposition~\ref{GHisom} in several situations, as the following four examples show. 
\begin{ex}\normalfont
\label{ex:dim2_DH}
Let $p$ be a prime number. Assume $d=2$ and  $G= D_{2p}$ is the dihedral group of order $2p$. Then $N=p$ and thus $|G|=2p < 4 N$. Proposition~\ref{GHisom} shows that the groups $G$ and $H$ are isomorphic. Recall that the only known finite groups in dimension two are $D_{2p}$ for $p\in\{2,3,4,5\}$, see \cite{BMMi-10,KaYa-15}.
\end{ex}

\begin{ex}\normalfont
Suppose $G$ is the permutation group $\mathfrak{S}_{d+1}$. Such a situation occurs in dimension three, see \cite[Tab.~1]{BaKaYa-16}; it also arises naturally when considering the model of non-intersecting lattice paths in arbitrary dimension, see the forthcoming Example~\ref{ex:ex_Coxeter_A}. In this case, we have $|G| < N 2^d$ for all $d\geq 2$, and thus by Proposition~\ref{GHisom}, the groups $G$ and $H$ are isomorphic.

Indeed, if $d=2$, then $|G|=6$, $N=3$ and it holds that $|G| < N 2^2$. If $d=3$ then $|G|= 24$, $N=4$ and  $|G| < N 2^3$. Finally if $d \geq 4$, $|G|= (d+1)!$ and $N= \frac{(d+1)!}{2}$ so that again  $|G| < N 2^d$. Then we conclude that $G$ and $H$ are isomorphic. 
\end{ex}

\begin{ex}\normalfont
Assume here that $d=4$ and $G=\mathfrak{S}_{3} \times \left(\frac{\mathbb Z}{2 \mathbb Z} \right)^2$. Some examples of models corresponding to this situation can be found in \cite[Tab.~1]{BuHoKa-21}. Then $G=|24|$. Since $N \geq 2$, it holds that $|G|< N 2^4$ and thus again $G$ and $H$ are isomorphic. 
 \end{ex}
 
 \begin{ex}\normalfont 
 Assume that $d=4$ and $G=\mathfrak{S}_{3} \times \mathfrak{S}_{3}$. Some examples of models that fit this situation are in \cite[Tab.~2]{BuHoKa-21}. Then $|G|= 36$. We claim that 
    $N \geq 3$. Indeed otherwise, $N =2$ and  there exists $(x,y) \in \mathfrak{S}_{3} \times \mathfrak{S}_{3}$ which has order $2$ and such that $\langle(x,y)\rangle$ is a normal subgroup of $G$. Clearly, this would imply that $\langle x\rangle$ is a normal subgroup of $\mathfrak{S}_{3}$, which is impossible since $\mathfrak{S}_{3}$ has no normal subgroup of order $2$. We conclude that $G$ and $H$ are isomorphic.
    \end{ex}

\subsection{A tool to determine \texorpdfstring{$\boldsymbol{G}$}{G} when \texorpdfstring{$\boldsymbol{H}$}{H} is known} 
\label{toolG}
We first establish the following. 
\begin{prop} 
\label{presentation}
Given any pair $(i,j)$ in $\{1,\ldots,d\}^2$, define $m_{i,j}$ as the orders of $\varphi_i \varphi_j$ in $G$. Let $K_d$ be the Coxeter group spanned by $(a_1,\ldots,a_d)$ and defined by the presentation
$$\mathcal{R} =\big\{(a_i a_j)^{m_{i,j}} |1 \leq i\leq j\leq d\}.$$
Then $|K_d| \geq  |G| \geq |H|$.
In particular, if $|K_d| = |H|$, then $G$ and $H$ are isomorphic.
\end{prop}
Note that it is possible to define $K_d$ because $m_{i,i}=1$ for all $i$.
\begin{proof}
We already know from Corollary~\ref{cor:imJH} that $|H| \leq |G|$. From the definition of a presentation, $K_d$ is isomorphic to the quotient of the free group $F_d$ of rank $d$   by the normal closure of $\mathcal{R}$ in $F_d$ (i.e., the smallest normal subgroup of $F_d$ containing $\mathcal{R}$). 

If $(G| \mathcal{S})$ is a presentation of $G$ (written with the generators $(\varphi_1,\ldots,\varphi_d)$ of $G$), then, since $(\varphi_1,\ldots, \varphi_d)$ satisfy the relations in $\mathcal{R}$, $(G | \mathcal{S}')$ is also a presentation of $G$, where we noted $\mathcal{S}'= \mathcal{S} \cup \mathcal{R}$.  Now again by the definition of a presentation, $G$ is  isomorphic to the quotient of the free group $F_d$    by the normal closure of  $\mathcal{S'}$ in $F_d$. Since $\mathcal{R} \subset \mathcal{S'}$, their normal closures satisfy the same inclusion in $F_d$ and we deduce that $|G| \leq|K_d|$. 
\end{proof}

We show how to apply the above result.
\begin{ex}\normalfont
Consider on $\R^d$ the following model 
\begin{equation*}
    \chi_\mathcal S(x_1,\ldots,x_d) = \overline{x_1} + \sum_{i=1}^{d-1} x_i \overline{x_{i+1}} + x_d.
\end{equation*}
In dimension $2$ this model is often called the tandem walk \cite{BMMi-10}. The latter model can thus be interpreted as a $d$-dimensional tandem model. It can also be interpreted as a possible model of non-intersecting lattice paths, where each coordinate represents the difference between two successive walks (see Example~\ref{ex:ex_Coxeter_A} for a closely related model). 

We calculate that $\boldsymbol{x}_0=(1,\ldots,1)$ and the covariance matrix $\Delta=\bigl(a_{i,j}\bigr)$  is given by 
 $$ a_{i,j} = \left| \begin{array}{rll}
 1 & \hbox{if} & i=j, \\
 -\frac{1}{2} & \hbox{if} & |i-j|=1, \\
 0 & \multicolumn{2}{l}{\textnormal{otherwise.}}
\end{array} \right.$$
From the classification of finite  Coxeter groups, $H$ is isomorphic to the permutation group $\mathfrak{S}_{d+1}$ and  $(r_1,\ldots,r_d)$ is a Coxeter system. As easily computed, the orders $m_{i,j}$ of $\varphi_i \varphi_j$ and $r_ir_j$ are the same and equal to 
$$m_{i,j} = \left| \begin{array}{rll}
 1 & \hbox{if} & i=j, \\
3 & \hbox{if} & |i-j|=1, \\
 2 & \multicolumn{2}{l}{\textnormal{otherwise.}}
 \end{array} \right.$$
 We conclude that $H \equiv \mathfrak{S}_{d-1}$ is isomorphic to $K_d$ and from Proposition \ref{presentation} to $G$.
 \end{ex}

As a concluding remark, if $|K_d|=  \infty$, Proposition~\ref{presentation} does not give any information. A more general result can be obtained in the same way: 
\begin{prop}
 \label{presentation2} 
Assume that $\mathcal{R}$ is a set of relations of the generators $(r_1,\ldots,r_d)$ of $H$ such that $(H|\mathcal{R})$ is a presentation of $H$. If for all $r_{i_1}\ldots r_{i_n} \in \mathcal{R}$, $\varphi_{i_1} \cdots \varphi_{i_n}=\Id$ in $G$, then $G$ and $H$ are isomorphic. 
 \end{prop}
\begin{proof}
The proof uses the same argument as in the proof of Proposition~\ref{presentation}: if $(G| \mathcal{S})$ is a presentation of $G$ (written with the generators $(\varphi_1,\ldots,\varphi_d)$ of $G$),   $(G | \mathcal{S}')$ with $\mathcal{S}'= \mathcal{S} \cup \mathcal{R}$ is also a presentation of $G$, which implies $|G| \leq |H|$. Again, we deduce from Corollary \ref{cor:imJH} that $G$ and $H$ are isomorphic. 
\end{proof}

\subsection{Dimension two}
Although our results do not bring any novelty in dimension two, we briefly recall what is known about the group $G$, which is a dihedral group, finite or infinite. In the unweighted case, if finite, $G$ can be of order $4$, $6$ or $8$, see \cite{BMMi-10}. If non-trivial weights are allowed, then the group can be of order $10$, see \cite[Sec.~7]{KaYa-15}, and it is believed that no higher order is possible.

On the other hand, the group $H$ can be any finite dihedral group. For example, consider any fixed value of $n\geq 3$ and define the transition probabilities $w(1,0)=w(-1,0)=\frac{\sin^2(\frac{\pi}{n})}{2}$ and $w(1, -1)=w(-1,1)=\frac{\cos^2(\frac{\pi}{n})}{2}$, then $H$ can be proved to be dihedral of order $2n$ using the computations in Section~\ref{subsec:illustration_2}; this example is inspired by the work \cite{Ra-11}.

\subsection{Dimension three}
We consider the three examples represented on Figure~\ref{fig:stepsets}.

	\begin{figure}
	\begin{center}
	\begin{tikzpicture}[scale=0.6]
			\draw[-stealth, line width=1pt, color=red] (1,1,1)--(0,0,2);
			\draw[-stealth, line width=1pt, color=red] (1,1,1)--(0,2,2);
			\draw[-stealth, line width=1pt, color=red] (1,1,1)--(0,0,1);
			\draw[-stealth, line width=1pt, color=red] (1,1,1)--(2,0,1);
			\draw[-stealth, line width=1pt, color=red] (1,1,1)--(1,2,1);
			\draw[-stealth, line width=1pt, color=red] (1,1,1)--(2,0,0);
			\draw[-stealth, line width=1pt, color=red] (1,1,1)--(2,2,0);
			
			\draw[opacity=0.3] (0,2,0)--(0,0,0)--(2,0,0);
			\draw (2,0,0)--(2,2,0)--(0,2,0);
			\draw (0,2,0)--(0,2,2);
			\draw (2,2,0)--(2,2,2);
			\draw (1,2,0)--(1,2,2);
			\draw (2,1,0)--(2,1,2);
			\draw (2,0,0)--(2,0,2);
			\draw[opacity=0.3] (0,1,0)--(0,1,2);
			\draw[opacity=0.3] (0,0,0)--(0,0,2);
			\draw[opacity=0.3] (1,0,0)--(1,0,2);
			\draw[opacity=0.3] (1,1,0)--(1,1,2);
			\draw[opacity=0.3] (1,0,1)--(1,2,1);
			\draw[opacity=0.3] (0,1,1)--(2,1,1);
			\draw[opacity=0.3] (1,0,0)--(1,2,0);
			\draw[opacity=0.3] (0,1,0)--(2,1,0);
			\draw (1,0,2)--(1,2,2);
			\draw (0,1,2)--(2,1,2);

				\draw[opacity=0.3] (0,2,1)--(0,0,1)--(2,0,1);
				\draw (2,0,1)--(2,2,1)--(0,2,1);

				\draw (0,0,2)--(2,0,2)--(2,2,2)--(0,2,2)--cycle;
		\end{tikzpicture}\qquad\qquad
		\begin{tikzpicture}[scale=0.6]
			\draw[-stealth, line width=1pt, color=red] (1,1,1)--(1,1,2);
			\draw[-stealth, line width=1pt, color=red] (1,1,1)--(1,2,2);
			\draw[-stealth, line width=1pt, color=red] (1,1,1)--(2,1,2);
			\draw[-stealth, line width=1pt, color=red] (1,1,1)--(1,2,1);
			\draw[-stealth, line width=1pt, color=red] (1,1,1)--(0,1,1);
			\draw[-stealth, line width=1pt, color=red] (1,1,1)--(2,0,1);
			\draw[-stealth, line width=1pt, color=red] (1,1,1)--(0,0,0);
			\draw[-stealth, line width=1pt, color=red] (1,1,1)--(0,2,0);
			\draw[-stealth, line width=1pt, color=red] (1,1,1)--(2,0,0);
			
			\draw[opacity=0.3] (0,2,0)--(0,0,0)--(2,0,0);
			\draw (2,0,0)--(2,2,0)--(0,2,0);
			\draw (0,2,0)--(0,2,2);
			\draw (2,2,0)--(2,2,2);
			\draw (1,2,0)--(1,2,2);
			\draw (2,1,0)--(2,1,2);
			\draw (2,0,0)--(2,0,2);
			\draw[opacity=0.3] (0,1,0)--(0,1,2);
			\draw[opacity=0.3] (0,0,0)--(0,0,2);
			\draw[opacity=0.3] (1,0,0)--(1,0,2);
			\draw[opacity=0.3] (1,1,0)--(1,1,2);
			\draw[opacity=0.3] (1,0,1)--(1,2,1);
			\draw[opacity=0.3] (0,1,1)--(2,1,1);
			\draw[opacity=0.3] (1,0,0)--(1,2,0);
			\draw[opacity=0.3] (0,1,0)--(2,1,0);
			\draw (1,0,2)--(1,2,2);
			\draw (0,1,2)--(2,1,2);

				\draw[opacity=0.3] (0,2,1)--(0,0,1)--(2,0,1);
				\draw (2,0,1)--(2,2,1)--(0,2,1);

				\draw (0,0,2)--(2,0,2)--(2,2,2)--(0,2,2)--cycle;
		\end{tikzpicture}\qquad\qquad
		\begin{tikzpicture}[scale=0.6]
			\draw[-stealth, line width=1pt, color=red] (1,1,1)--(1,2,2);
			\draw[-stealth, line width=1pt, color=red] (1,1,1)--(2,2,2);
			\draw[-stealth, line width=1pt, color=red] (1,1,1)--(0,0,2);
			\draw[-stealth, line width=1pt, color=red] (1,1,1)--(2,0,2);
			\draw[-stealth, line width=1pt, color=red] (1,1,1)--(0,2,1);
			\draw[-stealth, line width=1pt, color=red] (1,1,1)--(0,0,1);
			\draw[-stealth, line width=1pt, color=red] (1,1,1)--(0,1,0);
			\draw[-stealth, line width=1pt, color=red] (1,1,1)--(1,2,0);
			\draw[-stealth, line width=1pt, color=red] (1,1,1)--(2,1,0);
			\draw[-stealth, line width=1pt, color=red] (1,1,1)--(2,0,0);
			
			\draw[opacity=0.3] (0,2,0)--(0,0,0)--(2,0,0);
			\draw (2,0,0)--(2,2,0)--(0,2,0);
			\draw (0,2,0)--(0,2,2);
			\draw (2,2,0)--(2,2,2);
			\draw (1,2,0)--(1,2,2);
			\draw (2,1,0)--(2,1,2);
			\draw (2,0,0)--(2,0,2);
			\draw[opacity=0.3] (0,1,0)--(0,1,2);
			\draw[opacity=0.3] (0,0,0)--(0,0,2);
			\draw[opacity=0.3] (1,0,0)--(1,0,2);
			\draw[opacity=0.3] (1,1,0)--(1,1,2);
			\draw[opacity=0.3] (1,0,1)--(1,2,1);
			\draw[opacity=0.3] (0,1,1)--(2,1,1);
			\draw[opacity=0.3] (1,0,0)--(1,2,0);
			\draw[opacity=0.3] (0,1,0)--(2,1,0);
			\draw (1,0,2)--(1,2,2);
			\draw (0,1,2)--(2,1,2);

				\draw[opacity=0.3] (0,2,1)--(0,0,1)--(2,0,1);
				\draw (2,0,1)--(2,2,1)--(0,2,1);

				\draw (0,0,2)--(2,0,2)--(2,2,2)--(0,2,2)--cycle;
		\end{tikzpicture}
	\end{center}
	\caption{Left picture: the model considered in Example~\ref{ex:ex2}. Second picture: the model of Example~\ref{ex:ex3}. Right picture: the model presented on Example~\ref{ex:ex4}.}
	\label{fig:stepsets}
	\end{figure}
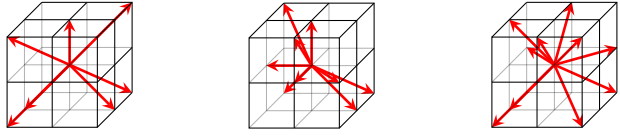

 \begin{ex}\normalfont
 \label{ex:ex2}
 Consider the following model
	$$\chi_{\mathcal S}(x,y,z)= \overline{x}\overline{y}\overline{z} + \overline{x}y\overline{z} + \overline{x}\overline{y} + y + x\overline{y} + x\overline{y}z + xyz,$$
as shown on Figure~\ref{fig:stepsets}.
To illustrate the objects we introduced in Section~\ref{sec:surjective}, we compute
$\boldsymbol{x_0}=(1,\frac{2}{\sqrt{3}},1)$ 
and
	\begin{equation*}
	    \Delta = \left(\begin{array}{ccc}
			1 & 0 & \frac{\sqrt{70}}{10}\\
			0 & 1 & 0\\
			\frac{\sqrt{70}}{10} & 0 & 1
		\end{array}\right) 
		= \left(\begin{array}{ccc}
			0 & 1 & 1\\
			1 & 0 & 0\\
			0 & 1 & -1
		\end{array}\right)
		\left(\begin{array}{ccc}
			1 & 0 & 0\\
			0 & 1+\frac{\sqrt{70}}{10} & 0\\
			0 & 0 & 1-\frac{\sqrt{70}}{10}
		\end{array}\right)
		\left(\begin{array}{ccc}
			0 & 1 & 0\\
			\frac{1}{2} & 0 & \frac{1}{2}\\
			\frac{1}{2} & 0 & -\frac{1}{2}
		\end{array}\right).
	\end{equation*}
Proposition \ref{prop:app} does not apply here since $H$ has a dihedral factor. 	 From the above diagonalization of $\Delta$, we can easily compute the domain $T$ in \eqref{eq:def_new_domain}.
Notice that 
	\begin{equation*}
		S_1 = J \varphi_1=\left(\begin{array}{ccc}
			-1 & 0 & -\frac{7}{5}\\
			0 & 1 & 0\\
			0 & 0 & 1
		\end{array}\right),\quad
		S_2= J \varphi_2=\left(\begin{array}{ccc}
			1 & 0 & 0\\
			0 & -1 & 0\\
			0 & 0 & 1
		\end{array}\right),\quad
		S_3 = J \varphi_3=\left(\begin{array}{ccc}
			1 & 0 & 0\\
			0 & 1 & 0\\
			-2 & 0 & -1
		\end{array}\right),
	\end{equation*}
so that we can compute 
$(S_1S_2)^2=(S_2S_3)^2=\Id_3$, while 
$$S_1S_3 = \left(\begin{array}{ccc}
			\frac{9}{5} & 0 & \frac{7}{5}\\
			0 & 1 & 0\\
			-2 & 0 & -1
		\end{array}\right)$$ 
		has eigenvalues
$1$ and $\frac{2\pm i \sqrt{21}}{5}$. 
The cosine associated with the second eigenvalue is $\frac{2}{5}$ and therefore rational. According to Niven's theorem \cite{Niv-33}, the second eigenvalue is of infinite order.  The fixed point method shown earlier \cite{DuHoWa-16} allows us to conclude that $G$ is infinite, while by Corollary~\ref{cor:imJH} we also conclude that $H$ is infinite (since $\Ima J$ is). 
\end{ex}

\begin{ex}\normalfont
\label{ex:ex3}
Consider the model
\begin{equation*}
    \chi_\mathcal S (x,y,z) = \overline{x} \overline{y} z + \overline{x} + \overline{x}yz + \overline{z} + \overline{z}y + y + x\overline{y} + x\overline{y}z + x\overline{z},
\end{equation*}
as shown in the left display on Figure~\ref{fig:stepsets}. One easily obtains $\boldsymbol{x_0}=(1,1,1)$, because the drift of the model is zero, and the covariance matrix is equal to
\begin{equation*}
    \Delta = \left(\begin{array}{rrr}
			1 & -\frac{1}{3} & -\frac{1}{3}\\
			-\frac{1}{3} & 1 & -\frac{1}{3}\\
			-\frac{1}{3} & -\frac{1}{3} & 1
		\end{array}\right).
\end{equation*}
Proposition~\ref{prop:app} immediately implies that the group $G$ is infinite. Note that this model belongs to the set $G_1$ in \cite[Tab.~1]{KaWa-17}, which means that there is no non-trivial relation between the generators $\varphi_1,\varphi_2,\varphi_3$ in \eqref{eq:group_G_def}.
\end{ex}

\begin{ex}\normalfont
\label{ex:ex4}
We now look at the model on the second display in Figure~\ref{fig:stepsets}
\begin{equation*}
    \chi_\mathcal S(x,y,z) = \overline{x}\overline{y}\overline{z}+\overline{x}\overline{y}+\overline{x}z+\overline{x}{y}+{y}\overline{z}+y z+{x}\overline{y}\overline{z}+{x}\overline{y}z+xz+{xy}\overline{z},
\end{equation*}
which also has zero drift, hence $\boldsymbol{x_0}=(1,1,1)$, and whose covariance matrix is the identity, as already observed in 
\cite[Sec.~5.4]{BoPeRaTr-20}. Accordingly, the reflection group $H$ is finite and isomorphic to $\bigl(\frac{\Z}{2\Z}\bigr)^3$. However, this model is known to admit an infinite group by \cite{DuHoWa-16}.
\end{ex}

\part{Eigenvalues of polyhedral nodal domains}
\label{part:2}

In the first part we saw that the asymptotics of the number of excursions $e_C(P,Q;n)$ in the orthant $C=\mathbb R_+^d$ is strongly related to the principal eigenvalue $\lambda_1$ of a Dirichlet problem on the polyhedral domain $T$ given by \eqref{eq:def_T}, see \eqref{eq:DW_exponent}. It is therefore natural to ask which are the polyhedral domains for which it is possible to compute this eigenvalue in closed form. Let us recall some facts in dimensions two and three. 

Dimension two is special and allows a systematic computation of $\lambda_1$, just by explicitly solving the eigenvalue problem. See e.g.\ \cite[Ex.~2]{DeWa-15} and \cite[Sec.~2.3]{BoRaSa-14} for the implementation of these calculations.

Dimension three is more complicated, and in general (i.e.\ for generic parameters)\ it is not possible to compute the eigenvalue in closed form (just as it is not possible to compute the first eigenvalue of a generic triangle in the plane $\mathbb R^2$ for the Dirichlet Laplacian). However, a list of polyhedral domains for which $\lambda_1$ (and actually the whole spectrum)\ can be computed can be found in \cite{BeBe-80,Be-83} by Bérard and Besson. The paper \cite{BoPeRaTr-20} explores the connection between walks in the three-dimensional orthant and these particular polyhedral domains, and proposes several examples of models corresponding to the polyhedral domains found in \cite{BeBe-80,Be-83}.

In Part~\ref{part:2} we extend the results of Bérard and Besson to arbitrary dimensions. We compute the Dirichlet eigenvalue $\la_1$ of polyhedral domains which are also nodal domains of ${\mathbb{S}}^{d-1}$, and we classify all such domains. We show that they must be the intersection of a chamber of a finite Coxeter group with $\mathbb S^{d-1}$. 

Part~\ref{part:2} consists of three sections. First, the main Theorem~\ref{pndomains} is stated and proved in Section~\ref{sec:polyhedral}. In Section~\ref{classification} we apply Theorem~\ref{pndomains} to the case of small dimensions two, three and four and completely classify the polyhedral nodal domains. In this way we recover the existing results of Bérard and Besson \cite{BeBe-80,Be-83} in dimension three and of Choe and Soret \cite{ChSo-09} in dimension four.
Finally, in Section~\ref{sec:three_ex} we give three examples that connect the first and second parts of our work. We consider some models of walks in arbitrary dimension and show how to explicitly compute their first eigenvalue $\lambda_1$ and asymptotic exponent $\alpha$ in \eqref{eq:DW_exponent}.

\section{Polyhedral nodal domains and their principal eigenvalues} 
\label{sec:polyhedral}

Let $d \geq 2$ and $({\mathbb{S}}^{d-1},\si_{d-1})$ be the $(d-1)$-dimensional sphere
\begin{equation*}
    {\mathbb{S}}^{d-1}= \{(x_1,\ldots,x_d) \in \R^d :  x_1^2+\cdots +x_d^2= 1\},
\end{equation*}
equipped with its natural Riemannian metric $\si_{d-1}$ obtained as the restriction of the Euclidean metric of $\R^d$. 
\begin{defn}
Let $U \subset \mathbb{S}^{d-1}$ be an open set. We say that $U$ is a polyhedral domain  if 
 \begin{equation}
     \label{eq:def_polyhedral_domain}
     U ={{\mathbb{S}}}^{d-1} \cap H_1^+ \cap \cdots \cap H_r^+,
\end{equation}
where for all $i$, $H_i^+$ is a half-space of $\R^d$ whose boundary is a linear hyperplane $H_i$. 
If $U$ is a polyhedral domain, the number $r$, which is assumed to be minimal in \eqref{eq:def_polyhedral_domain}, is the number of sides of~$U$.
\end{defn}

See Figure~\ref{fig:action_Delta} for examples of polyhedral domains in two and three dimensions. The domain $T$ in \eqref{eq:def_polyhedral_domain} as well as the orthant $\mathbb R_+^d$ are other examples of polyhedral domains. 

We denote by $\Delta_d$ the Laplacian on $\R^d$ and by $\D$ the Laplace-Beltrami operator on ${\mathbb{S}}^{d-1}$, see \cite[Sec.~3.2.3]{La-15}. 
We are interested in  polyhedral domains which are also nodal domains of ${\mathbb{S}}^{d-1}$, i.e., for which there exists 
$\phi \in \mathcal C^{\infty}({\mathbb{S}}^{d-1})$ such that $\phi$ is an eigenfunction of $-\D$ which satisfies $\phi>0$ on 
$U$ and 
$\phi=0$ on $\partial U= \cup_{i=1}^r (H_i \cap \overline{U})$. Note that, by \cite[Prop.~4.5.8]{La-15}, since $\phi$ has a constant sign on $U$, $\phi$ is then an eigenfunction associated to the 
first  eigenvalue $\la_1(U)$  for the Dirichlet problem on $(U,\si_{d-1})$. 
In the following, if $H$ is a hyperplane of $\R^d$, we denote by $s_H$ the orthogonal reflection with respect to $H$. We  prove: 

\begin{thm}
\label{pndomains}
Let  $U ={\mathbb{S}}^{d-1} \cap H_1^+ \cap \cdots \cap H_r^+$ be a polyhedral domain as in \eqref{eq:def_polyhedral_domain}. Then 
 $U$ is nodal if and only if there exists a finite set  $\Lambda$ of hyperplanes such that 
\begin{itemize}
 \item $U$ is a connected component of 
$ {\mathbb{S}}^{d-1} \cap \bigl( \cup_{H \in \Lambda} H \bigr)$; 
\item the Coxeter group $W:= \langle s_{H} | H \in \Lambda  \rangle \subset O(d)$ is finite and  acts on $\Lambda$.
\end{itemize}
Moreover, if $U$ satisfies the conditions above, let $k:=\sharp \Lambda$. Then the first eigenvalue of $U$ for the Dirichlet problem  is
\begin{equation*}
    \la_1(U)=k(d-2+k).
\end{equation*}
\end{thm}

We recall that finite Coxeter groups are reflection groups (see Appendix~\ref{sec:coxeter}).  These groups have been extensively studied and are all classified.  This allows us to classify all polyhedral nodal domains and compute their associated first eigenvalue. We give the complete list for $d =2, 3,4$ in Section~\ref{classification}, i.e., we give the complete list of polyhedral nodal domains of ${\mathbb{S}}^1$, ${\mathbb{S}}^2$ and ${\mathbb{S}}^3$.

\subsection{Preliminary results} 
Before proving Theorem~\ref{pndomains}, we need some preliminary results. We first recall a well-known result  (see for instance \cite[Sec.~5.1.3]{La-15}).
\begin{thm} \label{eig_sphere} 
Assume that $\phi \in \mathcal C^{\infty}({\mathbb{S}}^{d-1})$, $\phi \not\equiv 0$ satisfies 
$-\D\phi = \la \phi$. Then $\phi$ is the restriction to ${\mathbb{S}}^{d-1}$ of a homogeneous harmonic polynomial $P$ of $\R^d$. Moreover, if $k$ is the degree of $P$, then $\la = k(d-2+k)$.
\end{thm} 

We need the following lemma (the proof of which follows directly from the analyticity of $P_{\vert H}$):
\begin{lem} \label{cancellation}
Let $P$ be a polynomial on $\R^d$, $H \subset \R^d$ a hyperplane and $V$ an open set having a non-empty intersection with $H$. Assume that $P\equiv 0$ on $H\cap V$. Then $P \equiv 0$ on $H$.  
\end{lem}

Let $\Lambda$ be a finite set of hyperplanes  such that the Coxeter group $W:= \langle s_{H} | H \in \Lambda \rangle \subset O(d)$ is finite and  acts on $\Lambda$. For $g\in W$, denote $g\cdot P =P \circ g$.
 
\begin{defn} 
A polynomial $P$ on $\R^d$ is {\em antisymmetric} if for all $g \in W$, $g\cdot P = \det(g) P$. 
\end{defn}

For all $H \in \Lambda$, choose a unit vector $e_H$ orthogonal to $H$ and define $f_H=\langle e_H,\cdot\rangle$ so that $f_H$ is a linear form whose kernel is $H$. Define 
\begin{equation}
    \label{eq:def_P0}
    P_0:= \prod_{H \in \Lambda} f_H.
\end{equation}
Then, we will need the following  lemma which can also be found  in  \cite[Prop.~3]{BeBe-80}. 

\begin{lem}
\label{P0}
The polynomial $P_0$ in \eqref{eq:def_P0} is antisymmetric. Moreover, if $P$ is any antisymmetric polynomial, then $P_0$ divides $P$. 
\end{lem}

\begin{proof}
We start by proving the second statement. Take an orthonormal basis $(f_1,\ldots,f_d)$ such that $f_1 =e_H$ and $f_i \in H$ for all $i \geq 2$. Consider the associated coordinates $(y_1, \ldots, y_d)$ so that $f_H= y_1$. Then $P$ can be written as 
\begin{equation*}
    P =\sum_{i=0}^k y_1^i Q_i,
\end{equation*}
where for all $i$, $Q_i$ is a polynomial on $\R^d$ which does not depend on $y_1$. The antisymmetry property of $P$  means that 
$P(-y_1,y_2, \ldots,y_d)= - P(y_1,y_2, \ldots,y_d)$, which implies that $Q_0=0$ and thus $f_H = y_1$ divides 
$P$. Since for all $H \in \Lambda$, $f_H$ is irreducible and divides $P$,  and since the ring of polynomials on $\R^d$ is factorial, this means that $P_0$ divides $P$. 

Let us now prove that $P_0$ is antisymmetric. Let $H \in \Lambda$ and consider the set 
$$\Lambda_0:= \bigl\{H' \in \Lambda\setminus\{H\} |s_H(H')=H'\bigr\},$$ so that 
$H' \in \Lambda_0$  if and only if $H' \perp H$.
Define also $\Lambda'$ as a set of representatives of the orbits for the action of the group $\{\Id, s_H\}$ on $\Lambda \setminus \Lambda_0$. Then  
\begin{equation*}
    \bigl(\Lambda_0, \Lambda', s_H (\Lambda')\bigr)
\end{equation*}
is a partition of $\Lambda$. Let us write \eqref{eq:def_P0} under the form
\begin{equation*}
    P_0 = f_H\prod_{H' \in \Lambda_0} f_{H'}  \prod_{H' \in \Lambda'} \al_{H'},
\end{equation*}
where $\al_{H'} = f_{H'}f_{s_H(H')}$.
Then, note that 
\begin{itemize} 
 \item $s_H \cdot f_H= - f_H$,
 \item $s_H \cdot f_{H'} = f_{H'}$ if $H' \in \Lambda_0$,
 \item $s_H \cdot \al_{H'} = \al_{H'}$ if $H' \in \Lambda'$,
\end{itemize}
which implies that $s_H \cdot P_0=-P_0$ and thus $P_0$ is antisymmetric. 
\end{proof}

\subsection{Proof of Theorem~\ref{pndomains}}

Let $U$ be a polyhedral domain as in \eqref{eq:def_polyhedral_domain}. Assume that there exists a finite set $\Lambda$ of hyperplanes such that 
\begin{itemize}
   \item $U$ is a connected component of 
$ {\mathbb{S}}^{d-1} \cap \bigl( \cup_{H \in \Lambda} H \bigr)$; 
   \item the Coxeter group $W:= \left\langle s_{H} | H \in \Lambda  \right\rangle \subset O(d)$ is finite and acts on 
$\Lambda$.
\end{itemize}
Define  $P_0$ as in \eqref{eq:def_P0}. Since the Laplacian operator $\Delta_d$ commutes with the action of $W$, then for all $g \in W$, $g \cdot \Delta_d P_0= \Delta_d (g \cdot P_0)$. By Lemma~\ref{P0}, $P_0$ is antisymmetric and  we obtain that $\Delta_d P_0$ is also antisymmetric. Again by Lemma~\ref{P0}, $P_0$ must divide $\Delta_d P_0$ and thus must be $0$ since  $\deg(\Delta_d P_0) < P_0$. Hence, $P_0$ is a homogeneous harmonic polynomial and, by Theorem~\ref{eig_sphere}, its restriction to the sphere is an eigenfunction for the eigenvalue $k (d-2+k)$, where $k =\deg(P_0)= \sharp \Lambda$. Since $P_0$ vanishes only on $\cup_{H \in \Lambda} H$, it does not vanish on $U$ and thus $U$ is a nodal domain. 

\medskip

Conversely, assume that $U$ is a nodal domain as in \eqref{eq:def_polyhedral_domain}. Then there exists an eigenfunction $\phi$ of $\D$ on ${\mathbb{S}}^{d-1}$ which is positive on $U$ and vanishes on $\partial U$. By Theorem~\ref{eig_sphere}, $\phi$ is the restriction to ${\mathbb{S}}^{d-1}$ of a homogeneous harmonic polynomial $P$. By assumption, for all $i$, $P$ vanishes on $H_i \cap \overline{U}$ and since  $P$ is homogeneous,  vanishes on the cone 
\begin{equation*}
    \{ t (H_i \cap \overline{U} ) \vert t \in \R\},
\end{equation*}
whose interior in $H_i$ is not empty. By Lemma~\ref{cancellation}, $P$ vanishes on $H_i$. 

Let now $\Lambda$ be the set of hyperplanes on which $P$ is identically $0$ ($\Lambda$ contains the $H_i$ and possibly other hyperplanes). We first observe that $\Lambda$ is finite. Otherwise $P$ would vanish on an infinite number of hyperplanes and by analyticity, would vanish everywhere.  Let $W:= \langle s_H |H \in \Lambda\rangle$.
We prove that 
\begin{itemize}
 \item  {\em $W$ acts on $\Lambda$.} Given $H \in \Lambda$, consider the two open half-spaces $H^+$ and $H^-$ delimited by $H$. Define 
\begin{equation*}
    g : \left| \begin{array}{ccc} 
P +s_H \cdot P & \hbox{ on } & \overline{H^+}, \\
0 & \hbox{ on } & H^-. \end{array} \right. 
\end{equation*}
On $H$, $\partial_{e_H} (P +s_H\cdot P )=0$. Since $P +s_H \cdot P$ vanishes on $H$ and is harmonic, $\partial_{e_H}^2  (P +s_H\cdot P )=0$, which implies that $g$ is $\mathcal C^2$, harmonic and vanishes on a half-space; it must vanish everywhere. We deduce that $P + s_H \cdot P=0$. Since $H \in \Lambda$ is arbitrary and since the $\{s_H|H \in \Lambda\}$ is a system of generators of $W$, we get that for all $g \in W$, $g \cdot P=\det(g) P$, which means that $P$ is antisymmetric with respect to $W$. In particular, if $P$ vanishes everywhere on $H \in \Lambda$, it vanishes also everywhere on $g \cdot H$ with, by definition of $\Lambda$, implies that $g \cdot H \in \Lambda$. This proves that $W$ acts on $\Lambda$. 
\item {\em $W$ is finite.}  Assume $W$ is infinite and set $\Gamma= \cup_{H \in \Lambda} \{e_H,-e_H\}$. Then $W$ acts also on $\Gamma$. Since $W$ is infinite, there exists $e \in \Gamma$ such that the stabilizer 
\begin{equation*}
   W_e =W_{-e}=\{g \in W |g \cdot \pm e= \pm e\}    
\end{equation*}
is infinite.  Then $W_e$ acts on $\Gamma \setminus \{e ,-e\}$. By the same argument, there exists $e' \in \Gamma \setminus \{e ,-e\}$ such that the stabilizer
$
   (W_e)_{e'}= W_e \cap W_{e'}    
$
is infinite. By an immediate induction, we show that the intersection of all stabilizers $W_0= \cap_{e \in \Gamma} W_e$ is infinite. 
Now set $V =\mathrm{span}\Gamma$. Then $W_0$ acts on $V$ as $\{\Id\}$ (since it acts as  $\{\Id\}$ on a basis of $V$). Notice that if $x \in V^\perp$ then $x \in \cap_{H \in \Lambda} H$ and thus is invariant under the action of $W$ and thus of $W_0$. This means that $W_0$ acts on $V$ and $V^\perp$ as $\{\Id\}$  and thus $W_0 =\{ \Id\}$, which is a contradiction. 

\item {\em Conclusion.} Let $C= \cup_{t>0} tU$. Clearly $C$ is connected, does not intersect $\cup_{H \in \Lambda} H$ (since $P$ does not vanish on $C$) and any point of $\partial U$ belongs to $\cup_{H \in \Lambda} H$. This means that $C$ is a connected component of $\R^d \setminus \bigl( \cup_{H \in \Lambda}H\bigr) $, which implies that $U$ is a connected component of  $ {\mathbb{S}}^{d-1} \cap \bigl( \cup_{H \in \Lambda} H \bigr)$. 
\end{itemize}

\section{Classification of polyhedral nodal domains in small dimensions} \label{classification}

Let $U ={\mathbb{S}}^{d-1} \cap H_1^+ \cap \cdots \cap H_r^+$ be a polyhedral  domain with $r$ sides, as in \eqref{eq:def_polyhedral_domain}. It is completely characterized (up to an isometry) by the $\frac{r(r-1)}{2}$ angles $\al_{i,j}:= \widehat{(H_i,H_j)}$, $i < j$.
In this section, we give the exact list of angles for which $U$ is nodal. For any given polyhedral domain $U$, we will set   
\begin{equation*}
   a(U)=\bigl(\al_{1,2}, \ldots,\al_{1,r}, \al_{2,3},\ldots,\al_{2,r},\ldots ,\al_{r-1, r}\bigr).
\end{equation*} 
The following matrix representation of the (cosine of the)\ angles also appears in the literature (see e.g.\ \cite[Chap.~5]{Bo-68}):
\begin{equation*}
    \left(\begin{array}{ccccc}
    1& -\cos \alpha_{1,2}& \cdots & \cdots  &-\cos \alpha_{1,r}\\
    -\cos\alpha_{1,2}& 1 & -\cos\alpha_{2,3} &  \ldots & -\cos\alpha_{2,r}\\ \vdots &-\cos\alpha_{2,3} &\ddots  &\ddots & \vdots
    \\
        \vdots   & \cdots & \ddots & 1& -\cos\alpha_{r-1,r}\\
     -\cos\alpha_{1,r} & \cdots & \cdots & \cos \alpha_{r-1,r}  & 1
    \end{array}\right).
\end{equation*}
Interestingly, by Proposition~\ref{prop:angleHiHj} the above matrix corresponds exactly to our covariance matrix $\Delta$ in \eqref{eq:expression_covariance_matrix}.

We will say that a $\frac{r(r-1)}{2}$-tuple $a$ is admissible is there exists a polyhedral nodal domain $U$ such that  $a(U)=a$.   From Theorem~\ref{pndomains}, to each finite Coxeter group $W$ acting on $\R^d$ corresponds a unique polyhedral domain (uniqueness comes from the fact that all chambers are isometric). We then use the classification of irreducible Coxeter groups to classify these domains, as given in \cite[Chap.~2]{Hu-90}.

Let $W = W_1 \times\cdots \times W_m$ be a Coxeter group of $\R^d$. We use the decomposition $\R^d = (\oplus_{i=1}^m V_i)  \oplus Z$ of Appendix~\ref{sec:coxeter}; the quantity $r=d-\dim Z$ is the rank of $W$. 
Since in the whole paper (except in Example~\ref{ex:ex_Coxeter_A}), we deal with Coxeter groups of rank $d$ in $\R^d$, we reduce here the computation to the case where $W$ has rank $d$, i.e. 
\begin{equation}
\label{m=}
    \R^d = \oplus_{i=1}^m V_i
\end{equation}
and $U$ has exactly $r=d$ sides. In a second step, we will classify the $\frac{d(d-1)}{2}$-admissible tuples.

Observe that if $\si$  is a permutation of $\{1,\ldots, d\}$, then $a(U)=(\al_{i,j})$ is admissible if and only if $a(V)= (\al_{\si(i),\si(j)})$ is admissible, since the polyhedral domain $V$ can be obtained by $U$ by an isometry permuting the indices. In particular, in the classification below, we just choose one admissible tuple in each class under the action of the permutation group.

\medskip 

We will denote $k_1,\ldots,k_m$ the dimensions of the $V_i$ in \eqref{m=}, or equivalently, the ranks of the irreducible Coxeter groups $W_i$.  
  Note that, with the notation of Appendix~\ref{sec:coxeter}, it holds that $\sharp \Lambda= \sum_{i=1}^m \sharp \Lambda_i$. In the following, we use the classification of Coxeter groups given in \cite[Chap.~2]{Hu-90}. Note that in this classification the Coxeter group $A_d$ is isomorphic to the permutation group $\mathfrak{S}_{d+1}$. We use this several times in the paper.
  
\subsection{Dimension two}  
For $n=2$,  there is no need to use Theorem~\ref{pndomains}: up to an isometry, a nodal domain is just of the form 
$U_0={\mathbb{S}}^1$ or $U_k= \{(\cos(t),\sin(t)) \vert t \in (0, \frac{\pi}{k})\}$ for some $k \geq 1$. In the latter case, using the same notation as in Example~\ref{ex:dim2_DH}, the Coxeter group $W$ is the dihedral group $D_{2k}$ of order $2k$.

\subsection{Dimension three} 
Here $\frac{d(d-1)}{2}= 3$ and we give all admissible triplets and the corresponding $\la_1(U)$.  

\medskip

\begin{center}
\begin{tikzpicture}
	\begin{scope}[xshift=-8cm]
		\draw (-0.2,0) node[left]{$\frac{\Z}{2\Z}\times\frac{\Z}{2\Z}\times\frac{\Z}{2\Z}$};
		\draw (0,0) node{$\bullet$};
		\draw (1,0) node{$\bullet$};
		\draw (2,0) node{$\bullet$};
		\draw (2.2,0) node[right]{$\left(\frac{\pi}{2}, \frac{\pi}{2}, \frac{\pi}{2}\right)$};
	\end{scope}
	\begin{scope}[xshift=-8cm,yshift=-1cm]
		\draw (-0.2,0) node[left]{$\frac{\Z}{2\Z}\times D_{2k}$};
		\draw (0,0) node{$\bullet$};
		\draw (1,0) node{$\bullet$};
		\draw (2,0) node{$\bullet$};
		\draw[line width=1] (1,0)--(2,0);
		\draw (1.5,0) node[above]{$k$};
		\draw (2.2,0) node[right]{$\left(\frac{\pi}{2}, \frac{\pi}{2}, \frac{\pi}{k}\right)$};
	\end{scope}
	\draw (-0.2,0) node[left]{$A_3$};
	\draw (0,0) node{$\bullet$};
	\draw (1,0) node{$\bullet$};
	\draw (2,0) node{$\bullet$};
	\draw[line width=1] (0,0)--(1,0);
	\draw[line width=1] (1,0)--(2,0);
	\draw (2.2,0) node[right]{$\left(\frac{\pi}{3}, \frac{\pi}{2}, \frac{\pi}{3}\right)$};
	\begin{scope}[yshift=-1cm]
		\draw (-0.2,0) node[left]{$B_3$};
		\draw (0,0) node{$\bullet$};
		\draw (1,0) node{$\bullet$};
		\draw (2,0) node{$\bullet$};
		\draw[line width=1] (0,0)--(1,0);
		\draw[line width=1] (1,0)--(2,0);
		\draw (1.5,0) node[above]{$4$};
		\draw (2.2,0) node[right]{$\left(\frac{\pi}{3}, \frac{\pi}{2}, \frac{\pi}{4}\right)$};
	\end{scope}
	\begin{scope}[yshift=-2cm]
		\draw (-0.2,0) node[left]{$H_3$};
		\draw (0,0) node{$\bullet$};
		\draw (1,0) node{$\bullet$};
		\draw (2,0) node{$\bullet$};
		\draw[line width=1] (0,0)--(1,0);
		\draw[line width=1] (1,0)--(2,0);
		\draw (0.5,0) node[above]{$5$};
		\draw (2.2,0) node[right]{$\left(\frac{\pi}{5}, \frac{\pi}{2}, \frac{\pi}{3}\right)$};
	\end{scope}
\end{tikzpicture}

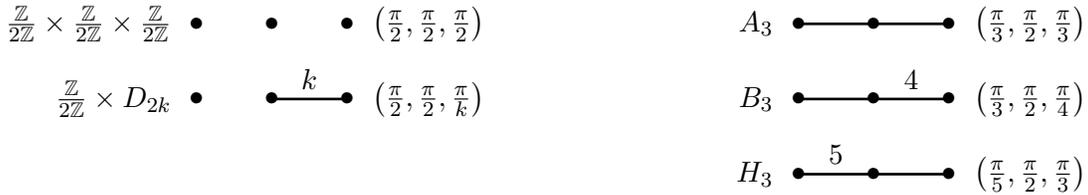
\captionof{figure}{Classification of reducible (on the left) and irreducible (on the right) Coxeter groups in dimension $d=3$. For each graph, the set of vertices is \( a(U) \), and two vertices \( \alpha \) and \( \beta \) are connected by an edge if \( m(\alpha, \beta) \geqslant 3 \), where \( m(\alpha, \beta) \) is the order of \( \alpha\beta \). In this case, the edge connecting \( \alpha \) to \( \beta \) is labeled with \( m(\alpha, \beta) \) (this label is omitted if \( m(\alpha, \beta) = 3 \)). Two vertices \( \alpha \) and \( \beta \) that are not connected satisfy \( m(\alpha, \beta) = 2 \). The Coxeter groups $A_3$ and $B_3$ belong to one-parameter families $A_d$ and $B_d$, which will be properly introduced in Examples~\ref{ex:ex_Coxeter_A} and \ref{ex:ex_Coxeter_B}, respectively.}
\label{classcox3}
\end{center}

We recall that $m$ is the number of irreducible components of $W$ (see~\eqref{m=})\ and $k_1,\ldots,k_m$ are the ranks of these irreducible components. See Figure~\ref{classcox3} for some important definitions. The following is a list of possible cases:
\begin{itemize} 
 \item $m=3$, then necessarily $k_1=k_2=k_3=1$. Then $a(U)= (\frac{\pi}{2}, \frac{\pi}{2}, \frac{\pi}{2})$,  $W=\frac{\Z}{2\Z}\times\frac{\Z}{2\Z}\times\frac{\Z}{2\Z}$ and $\sharp \Lambda=3$, so by Theorem~\ref{pndomains}, $\la_1(U)=12$.
 \item $m=2$, then (up to a permutation)\ $k_1=1$ and $k_2=2$. Then there exists $k \geq 2$ such that  $a(U)=(\frac{\pi}{2},  \frac{\pi}{2}, \frac{\pi}{k})$, $W=\frac{\Z}{2\Z}\times D_{2k}$ and $\sharp \Lambda= k+1$, so by Theorem~\ref{pndomains}, $\la_1(U)=(k+1)(k+2)$.
 \item $m=1$, $k_1=3$. Then, the list of admissible triplets are 
 \begin{itemize}
  \item $a(U) = (\frac{\pi}{3},\frac{\pi}{2},\frac{\pi}{3})$, $W=A_3$, $\sharp \Lambda=6$, so by Theorem~\ref{pndomains}, $\la_1(U)= 42$;
  \item $a(U)=(\frac{\pi}{3},\frac{\pi}{2},\frac{\pi}{4})$, $W=B_3$, $\sharp \Lambda=9$, so by Theorem~\ref{pndomains}, $\la_1(U)=90$;
  \item $a(U)=(\frac{\pi}{5},\frac{\pi}{2},\frac{\pi}{3})$, $W=H_3$, $\sharp \Lambda=15$, so by Theorem~\ref{pndomains}, $\la_1(U)=240$.
 \end{itemize}
\end{itemize}

\subsection{Dimension four}
\label{subsec:classif_Dim_4}
We shall use the classification of Coxeter groups in dimension four, which we now recall:
\begin{center}
	\begin{tikzpicture}
		\begin{scope}[xshift=-8cm]
			\draw (-0.2,0) node[left]{$\left(\frac{\Z}{2\Z}\right)^4$};
			\draw (0,0) node{$\bullet$};
			\draw (1,0) node{$\bullet$};
			\draw (2,0) node{$\bullet$};
			\draw (3,0) node{$\bullet$};
			\draw (3.2,0) node[right]{$\left(\frac{\pi}{2}, \frac{\pi}{2}, \frac{\pi}{2}, \frac{\pi}{2}, \frac{\pi}{2}, \frac{\pi}{2}\right)$};
		\end{scope}
		\begin{scope}[xshift=-8cm,yshift=-1cm]
			\draw (-0.2,0) node[left]{$\left(\frac{\Z}{2\Z}\right)^2 \times D_{2m}$};
			\draw (0,0) node{$\bullet$};
			\draw (1,0) node{$\bullet$};
			\draw (2,0) node{$\bullet$};
			\draw (3,0) node{$\bullet$};
			\draw[line width=1] (2,0)--(3,0);
			\draw (2.5,0) node[above]{$m$};
			\draw (3.2,0) node[right]{$\left(\frac{\pi}{2}, \frac{\pi}{2}, \frac{\pi}{2}, \frac{\pi}{2}, \frac{\pi}{2}, \frac{\pi}{m}\right)$};
		\end{scope}
		\begin{scope}[xshift=-8cm,yshift=-2cm]
			\draw (-0.2,0) node[left]{$D_{2m} \times D_{2m'}$};
			\draw (0,0) node{$\bullet$};
			\draw (1,0) node{$\bullet$};
			\draw (2,0) node{$\bullet$};
			\draw (3,0) node{$\bullet$};
			\draw[line width=1] (2,0)--(3,0);
			\draw[line width=1] (0,0)--(1,0);
			\draw (0.5,0) node[above]{$m$};
			\draw (2.5,0) node[above]{$m'$};
			\draw (3.2,0) node[right]{$\left(\frac{\pi}{m}, \frac{\pi}{2}, \frac{\pi}{2}, \frac{\pi}{2}, \frac{\pi}{2}, \frac{\pi}{m'}\right)$};
		\end{scope}
		\begin{scope}[xshift=-8cm,yshift=-3cm]
			\draw (-0.2,0) node[left]{$\frac{\Z}{2\Z} \times A_3$};
			\draw (0,0) node{$\bullet$};
			\draw (1,0) node{$\bullet$};
			\draw (2,0) node{$\bullet$};
			\draw (3,0) node{$\bullet$};
			\draw[line width=1] (2,0)--(3,0);
			\draw[line width=1] (1,0)--(2,0);
			\draw (3.2,0) node[right]{$\left(\frac{\pi}{2}, \frac{\pi}{2}, \frac{\pi}{2}, \frac{\pi}{3}, \frac{\pi}{2}, \frac{\pi}{3}\right)$};
		\end{scope}
		\begin{scope}[xshift=-8cm,yshift=-4cm]
			\draw (-0.2,0) node[left]{$\frac{\Z}{2\Z} \times B_3$};
			\draw (0,0) node{$\bullet$};
			\draw (1,0) node{$\bullet$};
			\draw (2,0) node{$\bullet$};
			\draw (3,0) node{$\bullet$};
			\draw[line width=1] (2,0)--(3,0);
			\draw[line width=1] (1,0)--(2,0);
			\draw (2.5,0) node[above]{$4$};
			\draw (3.2,0) node[right]{$\left(\frac{\pi}{2}, \frac{\pi}{2}, \frac{\pi}{2}, \frac{\pi}{3}, \frac{\pi}{2}, \frac{\pi}{4}\right)$};
		\end{scope}
		\begin{scope}[xshift=-8cm,yshift=-5cm]
			\draw (-0.2,0) node[left]{$\frac{\Z}{2\Z} \times H_3$};
			\draw (0,0) node{$\bullet$};
			\draw (1,0) node{$\bullet$};
			\draw (2,0) node{$\bullet$};
			\draw (3,0) node{$\bullet$};
			\draw[line width=1] (2,0)--(3,0);
			\draw[line width=1] (1,0)--(2,0);
			\draw (1.5,0) node[above]{$5$};
			\draw (3.2,0) node[right]{$\left(\frac{\pi}{2}, \frac{\pi}{2}, \frac{\pi}{2}, \frac{\pi}{5}, \frac{\pi}{2}, \frac{\pi}{3}\right)$};
		\end{scope}
		\draw (-0.2,0) node[left]{$A_4$};
		\draw (0,0) node{$\bullet$};
		\draw (1,0) node{$\bullet$};
		\draw (2,0) node{$\bullet$};
		\draw (3,0) node{$\bullet$};
		\draw[line width=1] (0,0)--(1,0);
		\draw[line width=1] (1,0)--(2,0);
		\draw[line width=1] (2,0)--(3,0);
		\draw (3.2,0) node[right]{$\left(\frac{\pi}{3}, \frac{\pi}{2}, \frac{\pi}{2}, \frac{\pi}{3}, \frac{\pi}{2}, \frac{\pi}{3}\right)$};
		\begin{scope}[yshift=-1cm]
			\draw (-0.2,0) node[left]{$B_4$};
			\draw (0,0) node{$\bullet$};
			\draw (1,0) node{$\bullet$};
			\draw (2,0) node{$\bullet$};
			\draw (3,0) node{$\bullet$};
			\draw[line width=1] (0,0)--(1,0);
			\draw[line width=1] (1,0)--(2,0);
			\draw[line width=1] (2,0)--(3,0);
			\draw (2.5,0) node[above]{$4$};
			\draw (3.2,0) node[right]{$\left(\frac{\pi}{3}, \frac{\pi}{2}, \frac{\pi}{2}, \frac{\pi}{3}, \frac{\pi}{2}, \frac{\pi}{4}\right)$};
		\end{scope}
		\begin{scope}[yshift=-2cm]
			\draw (-0.2,0) node[left]{$D_4$};
			\draw (0.5,0) node{$\bullet$};
			\draw (1.5,0) node{$\bullet$};
			\draw (2.5,0.5) node{$\bullet$};
			\draw (2.5,-0.5) node{$\bullet$};
			\draw[line width=1] (0.5,0)--(1.5,0);
			\draw[line width=1] (1.5,0)--(2.5,0.5);
			\draw[line width=1] (1.5,0)--(2.5,-0.5);
			\draw (3.2,0) node[right]{$\left(\frac{\pi}{3}, \frac{\pi}{2}, \frac{\pi}{2}, \frac{\pi}{3}, \frac{\pi}{3}, \frac{\pi}{2}\right)$};
		\end{scope}
		\begin{scope}[yshift=-3cm]
			\draw (-0.2,0) node[left]{$F_4$};
			\draw (0,0) node{$\bullet$};
			\draw (1,0) node{$\bullet$};
			\draw (2,0) node{$\bullet$};
			\draw (3,0) node{$\bullet$};
			\draw[line width=1] (0,0)--(1,0);
			\draw[line width=1] (1,0)--(2,0);
			\draw[line width=1] (2,0)--(3,0);
			\draw (1.5,0) node[above]{$4$};
			\draw (3.2,0) node[right]{$\left(\frac{\pi}{3}, \frac{\pi}{2}, \frac{\pi}{2}, \frac{\pi}{4}, \frac{\pi}{2}, \frac{\pi}{3}\right)$};
		\end{scope}
		\begin{scope}[yshift=-4cm]
			\draw (-0.2,0) node[left]{$H_4$};
			\draw (0,0) node{$\bullet$};
			\draw (1,0) node{$\bullet$};
			\draw (2,0) node{$\bullet$};
			\draw (3,0) node{$\bullet$};
			\draw[line width=1] (0,0)--(1,0);
			\draw[line width=1] (1,0)--(2,0);
			\draw[line width=1] (2,0)--(3,0);
			\draw (0.5,0) node[above]{$5$};
			\draw (3.2,0) node[right]{$\left(\frac{\pi}{5}, \frac{\pi}{2}, \frac{\pi}{2}, \frac{\pi}{3}, \frac{\pi}{2}, \frac{\pi}{3}\right)$};
		\end{scope}
	\end{tikzpicture}
	
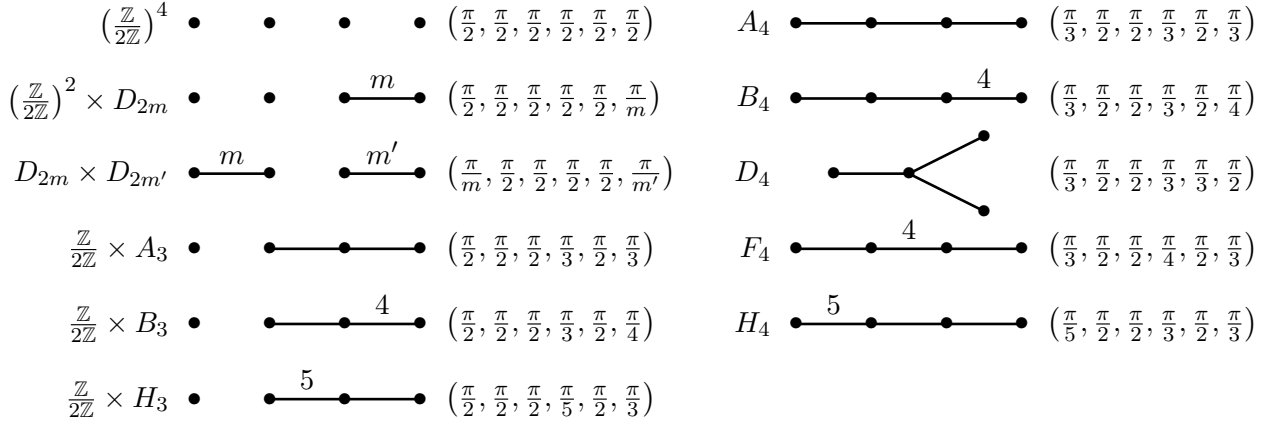
\captionof{figure}{Classification of reducible (on the left)\ and irreducible (on the right)\ Coxeter groups in dimension $d=4$.}\label{classcox4}
\end{center}
We now classify all admissible $6$-tuples (up to isometry) and recover the classification done in dimension four by Choe and Soret in \cite[Sec.~5]{ChSo-09}:
 \begin{itemize}
  \item $m=4$, $k_1=k_2=k_3=k_4=1$. Then $a(U)= (\frac{\pi}{2},\frac{\pi}{2},\frac{\pi}{2},\frac{\pi}{2},\frac{\pi}{2},\frac{\pi}{2})$, $W=\left(\frac{\Z}{2\Z}\right)^4$ and  $\sharp \Lambda= 4$, so by Theorem~\ref{pndomains}$, \la_1(U)=24$;
  \item $m=3$, $k_1=2, k_2=k_3=1$. Then $a(U)= (\frac{\pi}{2},\frac{\pi}{2},\frac{\pi}{2},\frac{\pi}{2},\frac{\pi}{2},\frac{\pi}{k})$, $k \geq 2$, $W=\left(\frac{\Z}{2\Z}\right)^2 \times D_{2k}$ and $\sharp \Lambda=k+2$, so by Theorem~\ref{pndomains}, $\lambda_1(U) = (k+2)(k+4)$;
\item $m=2$, $k_1=k_2=2$. Then $a(u) = (\frac{\pi}{k},\frac{\pi}{2},\frac{\pi}{2},\frac{\pi}{2},\frac{\pi}{2},\frac{\pi}{k'})$, $k,k' \geq 2$, $W=D_{2k} \times D_{2k'}$ and $\sharp \Lambda= k+k'$, so by Theorem~\ref{pndomains}, $\la_1(U)=(k+k')(k+k'+2)$;
\item $m=2$, $k_1=3$, $k_2=1$. Using the classification in dimension $3$ (see Figure~\ref{classcox3} on the right), we obtain the following cases: 
\begin{itemize}
  \item $a(U) = (\frac{\pi}{2},\frac{\pi}{2}, \frac{\pi}{2}, \frac{\pi}{2},\frac{\pi}{3},\frac{\pi}{3})$, $W\simeq A_3\times \frac{\Z}{2\Z}$ and $\sharp \Lambda=7$, so by Theorem~\ref{pndomains}, $\la_1(U)= 63$;
  \item $a(U)=(\frac{\pi}{2},\frac{\pi}{2}, \frac{\pi}{2},\frac{\pi}{2},\frac{\pi}{3},\frac{\pi}{4})$, $W\simeq B_3\times \frac{\Z}{2\Z}$ and $\sharp \Lambda=10$, so by Theorem~\ref{pndomains}, $\la_1(U)=120$;
  \item $a(U)=(\frac{\pi}{2},\frac{\pi}{2},\frac{\pi}{2},\frac{\pi}{2},\frac{\pi}{3},\frac{\pi}{5})$, $W \simeq H_3 \times \frac{\Z}{2\Z}$ and $\sharp \Lambda=16$, so by Theorem~\ref{pndomains}, $\la_1(U)=288 $;
\end{itemize}

  \item $m=1$, $k_1=4$. There are five possibilities: 
\begin{itemize}
   \item $a(U) = (\frac{\pi}{3},\frac{\pi}{2},\frac{\pi}{2},\frac{\pi}{3},\frac{\pi}{2},\frac{\pi}{3})$, $W= A_4$, $\sharp \Lambda=10$, $\la_1(U)= 120$; 
    \item $a(U) = (\frac{\pi}{3},\frac{\pi}{2},\frac{\pi}{2},\frac{\pi}{3},\frac{\pi}{2},\frac{\pi}{4})$, $W= B_4$, $\sharp \Lambda=16$, $\la_1(U)= 272$; 
    \item $a(U)=  (\frac{\pi}{3},\frac{\pi}{3},\frac{\pi}{3},\frac{\pi}{2},\frac{\pi}{2},\frac{\pi}{2})$, $W= D_4$, $\sharp \Lambda=12$, $\la_1(U)= 168$;  
    \item $a(U)=  (\frac{\pi}{3},\frac{\pi}{2},\frac{\pi}{2},\frac{\pi}{4},\frac{\pi}{2},\frac{\pi}{3})$, $W= F_4$, $\sharp \Lambda=24$, $\la_1(U)= 624$; 
    \item $a(U)= (\frac{\pi}{5},\frac{\pi}{2},\frac{\pi}{2},\frac{\pi}{3},\frac{\pi}{2},\frac{\pi}{3})$, $W= H_4$, $\sharp \Lambda=60$, $\la_1(U)= 3720$.
  \end{itemize}
\end{itemize}

\section{Three examples} 
\label{sec:three_ex}

To conclude this part, we give three examples where our results allow $\la_1$ to be computed explicitly. We recall (see Theorem~\ref{thm:DW_formula_exponent}) that $\la_1$ is the smallest eigenvalue of the Dirichlet problem~\eqref{eq:one-term-asymp}, and its explicit value gives the asymptotics \eqref{eq:one-term-asymp} of the number of excursions between two points. 
\begin{ex}\normalfont
Consider the $4$-dimensional model whose inventory \eqref{eq:inventory} is given by 
\begin{equation*}
    \chi_\mathcal S(x,y,z,w) = \overline{w}+x \overline{z}+ \overline{x}y + z+ \overline{y}w.
\end{equation*}
This is model~37 in \cite[Tab.~3]{BuHoKa-21}, whose group $G$ is proved in \cite{BuHoKa-21} to be isomorphic to the symmetry group $\mathfrak S_5$. We have $\boldsymbol{x_0}=(1,1,1,1)$ and a simple calculation gives 
\begin{equation*}\Delta = \left( \begin{array}{cccc} 
1 & - \frac{1}{2} & -\frac{1}{2}  & 0 \\
 -\frac{1}{2} & 1 &  0 &  - \frac{1}{2}  \\
-\frac{1}{2} & 0 & 1 &  0 \\
0 & -\frac{1}{2} &  0& 1 
\end{array}
\right). \end{equation*}
Using the notation of Section~\ref{classification}, the list of angles between the hyperplanes bounding $\Delta^{-\frac{1}{2}} \R^d_+$ is given by Proposition~\ref{prop:reformulation_cov_matrix} and its consequence~\eqref{angle}:  
$\left(\frac{\pi}{3}, \frac{\pi}{3},\frac{\pi}{2},\frac{\pi}{2},\frac{\pi}{3},\frac{\pi}{2} \right)$. The permutation of variables $(x,y,z,w) \to (z,x,y,w)$   gives rise to the list 
$$\left(\frac{\pi}{3},\frac{\pi}{2},\frac{\pi}{2},\frac{\pi}{3},\frac{\pi}{2}, \frac{\pi}{3}\right).$$
From the classification of Section~\ref{subsec:classif_Dim_4}, we obtain that $\la_1=120$. 

 This example is actually a particular case (up to a permutation of variables) of the $d$-dimensional example given in Section~\ref{toolG}. With the same argument, one checks that for the $d$-dimensional case, it holds that  $\la_1= \frac{1}{4}d(d+1)(d+4)(d-1)$. 
\end{ex}\normalfont

\begin{ex}\normalfont
\label{ex:ex_Coxeter_A}
Consider the simple $d$-dimensional random walk with jumps 
\begin{equation*}
    \chi_\mathcal S(x_1,\ldots,x_d)=x_1+\overline{x_1}+\cdots + x_d+\overline{x_d},
\end{equation*}
in the cone given by the Weyl chamber of type $A$, namely \begin{equation*}
    W_A=\{x_1<x_2<\cdots <x_d\}.
\end{equation*}
The asymptotics of the number of such walks (sometimes called non-intersecting lattice paths)\ is known and the exponent $\alpha$ in \eqref{eq:one-term-asymp} is given by $\alpha=\frac{d^2}{2}$; see e.g.\ \cite[Thm~2.1]{Fe-18},  \cite[Thm~1.1]{EiKo-08} and \cite[Thm~1]{DeWa-10}. Let us briefly explain how this relates to our results. 
First, we observe that the walls of $W_A$ are the hyperplanes $H_i$, $1 \leq i \leq d-1$ defined by the equations $x_i=x_{i+1}$. Define $u_i=e_i-e_{i+1}$, where $(e_1,\ldots,e_d)$ is the canonical basis of $\R^d$. Then $u_i \perp H_i$ and is clearly outward with respect to $W_A$. Since $(u_i,u_j)= \|u_i\| \|u_j\| \cos( \widehat{u_i u_j} )$, we obtain that 
\begin{equation} \label{WA} 
\widehat{H_i H_j}= \pi -\widehat{u_i u_j} = \left| \begin{array}{cl} 
\frac{\pi}{3} & \hbox{if }  |i-j|=1, \\
\frac{\pi}{2} & \hbox{otherwise}.
\end{array} \right. \end{equation}
From the classification of Coxeter groups, we get that $W_A$ is a chamber of the group $A_{d-1}= \mathfrak{S}_{d}$ which is a $(d-1)$-rank Coxeter group.
We use our Theorem~\ref{pndomains} to show that  
\begin{equation*}
    \la_1(U)=k(d-2+k),
\end{equation*}
where $k=\sharp \Lambda$ is the number of hyperplanes needed to define the Weyl chamber $W_A$, or equivalently the number of reflections in $A_d$, which is known to be $k=\frac{d(d-1)}{2}$. Using the formula~\eqref{eq:DW_exponent}, we immediately obtain $\alpha=\frac{d^2}{2}$.
\end{ex}

\begin{ex}\normalfont
\label{ex:ex_Coxeter_B}
Considering the same random walk as in Example~\ref{ex:ex_Coxeter_A}, but in the Weyl chamber of type~$B$
\begin{equation*}
    W_B=\{0<x_1<x_2<\cdots <x_d\},
\end{equation*}
this model has also been explored in the literature, see \cite[Thm~2.3]{KoSc-10} and \cite[Thm~5.1]{Fe-14}. It is also known as a model of $d$ vicious walkers models subject to a wall restriction. In particular, it is known that the exponent is $\alpha=d^2+\frac{d}{2}$.
The chamber $W_B$ has the same walls as $W_A$ above plus $H_0$ defined by $x_1=0$. Clearly, beside the relations~\eqref{WA}, we have, since $e_1$ is an inward  normal vector to $H_0$,  
\begin{equation*}
\widehat{H_0 H_i} = \widehat{e_1 u_i} = \left| \begin{array}{cl} 
\frac{\pi}{4} & \hbox{if }  i=1, \\
\frac{\pi}{2} & \hbox{otherwise}.
\end{array} \right.
\end{equation*}
We deduce that the associated Coxeter group is the $d$-rank group $B_d$. As in Example~\ref{ex:ex_Coxeter_A}, we use Theorem~\ref{pndomains} with $k=d^2$ (corresponding to the number of reflections in $ B_d$) and $\la_1(U)=k(d-2+k)$, and finally we use \eqref{eq:DW_exponent} and get the announced value of the exponent $\alpha$.
\end{ex}

\appendix
\section{Coxeter groups} 
\label{sec:coxeter} 
In this section we recall some facts about Coxeter groups used throughout our article; they can all be found in \cite{Hu-90}. 

\begin{defn}
\label{def:Coxeter_system}
Let $W$ be a group. 
\begin{itemize}
    \item A {\em Coxeter system} is a set $S \subset W$ of generators,   subject only to
    relations $(ss')^{m(s,s')}=1$ for all $s,s'\in S$. Here $m(s,s') \in \N$ satisfies $m(s,s)=1$ and if $s\not=s'$, $m(s,s') \geq 2$.
    \item The group $W$ is a {\em Coxeter group} if it admits a Coxeter system.  
    \end{itemize}
\end{defn}

Many results are known about Coxeter groups, but we focus on two results that play a crucial role in our paper: 

\begin{thm}[p.~16 in \cite{Hu-90}]
Let $V$ be a finite dimensional space with a scalar product $\langle \cdot,\cdot\rangle$, and let $\Lambda$ be a set of independent vectors. For each $v \in \Lambda$, let $s_v$ be the orthogonal reflection with respect to $\langle v\rangle^\perp$ and $W$ be the subgroup of $\bigl(V, \langle\cdot,\cdot\rangle \bigr) $ spanned by $\{s_v | v \in \Lambda\}$.  Then the reflection group $W$ is a Coxeter group. 
\end{thm}

\begin{thm}[p.~133 in \cite{Hu-90}]
Every Coxeter group can be realised as a reflection group (as in the statement of the above theorem). 
\end{thm}

Let $S$ be a Coxeter system of a finite Coxeter group $W$. All results in the following are, for example, given in \cite{Hu-90}. The integers $m(s,s')$ in Definition~\ref{def:Coxeter_system} completely determine $W$.  The classification of finite Coxeter groups implies the following
\begin{prop}
\label{possible_order}
If the Coxeter group $W$ has no irreducible component which is a dihedral group, then for all $s, s' \in S$, $m(s,s') \in \{2,3,4,5,6\}$.
\end{prop}

Let $\Lambda$ be a set of hyperplanes of $\R^d$ and $W$ be the group of isometries spanned by 
$\{s_H| H\in \Lambda\}$. We assume that $W$ is finite. Then, there exist subspaces $V_1, \ldots,V_m,Z$ of $\R^d$,  $W_1,\ldots,W_m$ finite irreducible Coxeter groups of $V_1,\ldots,V_m$, such that 
\begin{itemize}
 \item $\R^d = (\oplus_{i=1}^m V_i)  \oplus Z$ and the sum is orthogonal; 
 \item $W= W_1 \times \cdots \times W_m$;
 \item $W$ acts  on $Z$ as the identity;
 \item $W_i$ acts transitively on the set $\Lambda_i$ of hyperplanes $H_i$ of $V_i$ such that
  \begin{equation} \label{lambdai}
   H:= Z \times V_1 \times\cdots V_{i-1} \times H_i \times V_{i+1} \times \cdots \times V_m \in \Lambda.
  \end{equation}
\end{itemize}
In particular $\Lambda = \cup_{i=1}^m \Lambda_i$, where the hyperplanes of $V_i$ are identified to hyperplanes of $\R^d$ via Equation~\eqref{lambdai}.  
The rank of $W$ is by definition \begin{equation*}
   \sum_{i=1}^m \dim V_i=d -\dim Z.    
\end{equation*}
A connected component of $\R^d \setminus \bigl(\cup_{H \in \Lambda} H \bigr)$  is called a {\em chamber of $W$}, and the hyperplanes of $\Lambda$ which intersect the boundary of a chamber are called {\em walls}. With this terminology,  Theorem~\ref{pndomains} says that a polyhedral domain $U$ of ${\mathbb{S}}^{d-1}$ is nodal if and only if it is the intersection of the chamber of a finite Coxeter group $W$ with ${\mathbb{S}}^{d-1}$. A classical theorem asserts that the number of walls of a chamber is the rank of $W$. We thus deduce that the number of sides of $U$ must also be the rank of $W$ and thus must be smaller than $d$. Note also that all the chambers are isometric. Finally, in the present situation, we have:
\begin{thm} \label{chamber_gen} 
Let $C$ be a chamber of a finite Coxeter group and let $\Lambda_C$ be the set of hyperplanes bounding $C$. Then $S = \{s_H| H\in \Lambda_C\}$ is a Coxeter system of $W$.
\end{thm}

\section*{Acknowledgments}
We would like to thank Gérard Besson, Thomas Gobet, Jérémie Guilhot, Manuel Kauers and Cédric Lecouvey for interesting discussions. EH is supported by the project Einstein-PPF (\href{https://anr.fr/Project-ANR-23-CE40-0010}{ANR-23-CE40-0010}), funded by the French National Research Agency. KR is supported by the project RAWABRANCH (\href{https://anr.fr/Project-ANR-23-CE40-0008}{ANR-23-CE40-0008}), funded by the French National Research Agency. KR thanks
the VIASM (Hanoï, Vietnam)\ for their hospitality and wonderful working conditions.

\bibliographystyle{siam}

\end{document}